\documentclass[11pt, reqno]{amsart}
\pdfoutput=1
\usepackage{amsmath}
\usepackage{amssymb}
\usepackage{amsthm}
\usepackage{enumerate}
\usepackage{amscd}
\usepackage{graphicx}
\usepackage{tabu}
\usepackage[top=1.50in, bottom=1.50in, left=1.25in, right=1.25in]{geometry}
\usepackage[all]{xy}
\usepackage{longtable}
\usepackage[utf8]{inputenc}
\usepackage{listings}
\usepackage{filecontents}
\usepackage{scrextend}
\usepackage{url}
\usepackage[all]{xy}
\usepackage{stmaryrd}
\usepackage{bbm}
\usepackage{relsize}
\usepackage{mathrsfs}
\usepackage{hyperref,color}
\usepackage{dsfont}

\let\Im\relax
\DeclareMathOperator{\Fl}{\mathcal{F}\ell}

\DeclareMathOperator{\pr}{pr}

\DeclareMathOperator{\Im}{Im}

\DeclareMathOperator{\red}{red}

\DeclareMathOperator{\id}{id}

\DeclareMathOperator{\Vect}{Vect}

\DeclareMathOperator{\ad}{ad}

\DeclareMathOperator{\uHom}{\underline{Hom}}

\DeclareMathOperator{\GL}{GL}

\DeclareMathOperator{\SL}{SL}

\DeclareMathOperator{\Spec}{Spec}

\DeclareMathOperator{\Rep}{Rep}

\DeclareMathOperator{\Gr}{Gr}

\DeclareMathOperator{\Ext}{Ext}

\DeclareMathOperator{\IC}{IC}
\DeclareMathOperator{\rank}{rank}

\DeclareMathOperator{\sss}{ss}

\let\phi\varphi

\DeclareMathOperator{\nor}{nor}
\DeclareMathOperator{\CT}{CT}

\def\cB{\mathcal{B}}
\def\cF{\mathcal{F}}
\def\cL{\mathcal{L}}
\def\cO{\mathcal{O}}\def\cP{\mathcal{P}}
\def\cS{\mathcal{S}}

  \def\sP{\mathscr{P}}

\def\bbA{\mathbb{A}}
\def\bbF{\mathbb{F}}\def\bbG{\mathbb{G}}

\def\bbN{\mathbb{N}}\def\bbP{\mathbb{P}}
\def\bbQ{\mathbb{Q}}

\def\bbZ{\mathbb{Z}}

\def\fw{\mathfrak{w}}

\def\ra{\rightarrow}

\def\lra{\longrightarrow}

\def\lmapsto{\longmapsto}

\def\oA{\overline{A}}

\theoremstyle{definition}
\newtheorem{defn}{Definition}[section]
\newtheorem{exmp}[defn]{Example}

\newtheorem{rmrk}[defn]{Remark}
\newtheorem{nota}[defn]{Notation}
\newtheorem{defn*}{Definition}[section]
\newtheorem{exmp*}[defn*]{Example}
\newtheorem{exmps*}[defn*]{Examples}
\newtheorem{rmrk*}[defn*]{Remark}
\newtheorem{nota*}[defn*]{Notation}
\newtheorem{quest*}[defn*]{Question}

\theoremstyle{plain}
\newtheorem{thm}[defn]{Theorem}
\newtheorem{prop}[defn]{Proposition}
\newtheorem{lem}[defn]{Lemma}
\newtheorem{cor}[defn]{Corollary}
\newtheorem{thm*}[defn*]{Theorem}
\newtheorem{prop*}[defn*]{Proposition}
\newtheorem{lem*}[defn*]{Lemma}
\newtheorem{cor*}[defn*]{Corollary}

\newcommand\restr[2]{{
  \left.\kern-\nulldelimiterspace 
  #1 
  \vphantom{\big|} 
  \right|_{#2} 
  }}



\title{Constant term functors with $\bbF_p$-coefficients}
\author{Robert Cass and C\'edric P\'epin}
\address{Department of Mathematics, Harvard University}
\email{rcass@math.harvard.edu}
\address{D\'epartment de Math\'ematiques, Institut Galil\'ee, Universit\'e Paris 13}
\email{cpepin@math.univ-paris13.fr}

\numberwithin{equation}{section}

\begin{document}

\begin{abstract}
We study the constant term functor for $\bbF_p$-sheaves on the affine Grassmannian in characteristic $p$ with respect to a Levi subgroup. Our main result is that the constant term functor induces a tensor functor between categories of equivariant perverse $\bbF_p$-sheaves. We apply this fact to get information about the Tannakian monoids of the corresponding categories of perverse sheaves. As a byproduct we also obtain geometric proofs of several results due to Herzig on the mod $p$ Satake transform and the structure of the space of mod $p$ Satake parameters.
\end{abstract}

\maketitle

\thispagestyle{empty}

\setcounter{tocdepth}{1}
\tableofcontents

\section{Introduction}

\subsection{Constant term functors with $\overline{\mathbb{Q}}_\ell$-coefficients}
In the Langlands program over a global field $F$, the \emph{constant term} and \emph{Eisenstein series} operators relate automorphic functions with respect to a reductive group $G/F$ and its Levi subgroups. When $F$ is the function field of a smooth curve $C$ over a finite field $\bbF_q$ of characteristic $p$, it is possible to upgrade these operators to functors on sheaves, cf. \cite{BG02} and \cite{DG16}. 

For simplicity suppose $G$ arises from a split connected reductive group over $\bbF_q$. For each $x \in C(\bbF_q)$ a local Hecke algebra acts on 
automorphic functions. After choosing an isomorphism $\mathbb{C} \simeq \overline{\mathbb{Q}}_\ell$ and a uniformizing element at $x$, this local Hecke algebra can be identified with the unramified Hecke algebra $\mathcal{H}_{G,\ell}$ of $G(\mathbb{F}_q(\!(t)\!))$ with $\overline{\mathbb{Q}}_\ell$-coefficients. 

In order to geometrize $\mathcal{H}_{G,\ell}$, one considers the following functors on $\bbF_q$-algebras:
$$LG \colon R \mapsto G(R(\!(t)\!)), \quad  \quad L^+G \colon R \mapsto G(R[\![t]\!]).$$ The \emph{affine Grassmannian} is the fpqc-quotient 
$\Gr_G :=LG/L^+G$, which is representable by an ind-scheme. Then in the context of the geometric Langlands program, the algebra $\mathcal{H}_{G,\ell}$ is replaced by the tensor category 
$(P_{L^+G}(\Gr_G, \overline{\mathbb{Q}}_\ell), *)$ of $L^+G$-equivariant perverse $\overline{\mathbb{Q}}_\ell$-sheaves on $\Gr_G$ for $\ell \neq p$.

If $P$ is a parabolic subgroup of $G/ \bbF_q$ with Levi factor $L$ there is a diagram 
\begin{equation} \label{GrP}
\xymatrix{ 
&\Gr_P \ar[dl]_q \ar[dr]^p & \\
\Gr_L & & \Gr_G.
}
\end{equation}
The local analogue of the constant term functor is 
$$\xymatrix{\CT_{L,\ell}^G \colon P_{L^+G}(\Gr_G, \overline{\mathbb{Q}}_\ell) \ar[rr]^(.6){q_! \, \circ \, p^*[\deg_P]} & & D_{c}^b(\Gr_L, \overline{\mathbb{Q}}_\ell)}$$ for a certain locally constant function $\deg_P \colon \Gr_P \rightarrow \mathbb{Z}$, cf. (\ref{degPdef}). The function-sheaf dictionary sends $\CT_{L,\ell}^G$ to the Satake transform $\mathcal{H}_{G,\ell} \rightarrow \mathcal{H}_{L,\ell}$ up to a normalization factor. Remarkably, the functor $\CT_{L,\ell}^G$ takes values in $P_{L^+G}(\Gr_L, \overline{\mathbb{Q}}_\ell)$, and is compatible with the tensor structures. 

\subsection{Constant term functors with $\bbF_p$-coefficients}
Let $k$ be an algebraically closed field of characteristic $p>0$ and let $G$ be a connected reductive group over $k$. Let $\Gr_G$ be the affine Grassmannian of $G$ over $k$, and let $(P_{L^+G}(\Gr_G, \bbF_p), *)$ be the abelian symmetric monoidal category of $L^+G$-equivariant perverse $\bbF_p$-sheaves on $\Gr_G$ as defined in \cite{modpGr}. 

Fix a maximal torus and a Borel subgroup $T \subset B \subset G$. Let $B\subset P\subset G$ be a standard parabolic subgroup and $L$ be its Levi factor containing $T$.

\begin{defn}
The \emph{$L$-constant term functor} is
$$\xymatrix{
\CT_L^G := Rq_! \circ Rp^* [\deg_P] \colon P_{L^+G}(\Gr_G, \mathbb{F}_p) \rightarrow D_c^b(\Gr_L, \mathbb{F}_p).
}$$
\end{defn}

Our main result is the following, cf. \ref{section_reslevi}:

\begin{thm}  \label{thm1}
%[cf. \S \ref{section_reslevi}] 
The functor $\CT_L^G$ induces an exact faithful tensor functor
$$
\xymatrix{
\CT_L^G \colon (P_{L^+G}(\Gr_G,\bbF_p),*) \ar[r] & (P_{L^+L}(\Gr_{L},\bbF_p),*).
}
$$
\end{thm}

Let us start by explaining why $\CT_L^G$ preserves perversity. Let $X_*(T)$ be the group of cocharacters of $T$ and $X_*(T)^+$ (resp. $X_*(T)_-$) be the monoid of dominant (resp. anti\-do\-mi\-nant) cocharacters. For $\lambda \in X_*(T)^+$ let $\Gr_G^{\leq \lambda}$ be the reduced closure of the $L^+G$-orbit of  $\lambda(t)$ in $\Gr_G$. By \cite[1.5]{modpGr} the simple objects in $P_{L^+G}(\Gr_G, \bbF_p)$ are the shifted constant sheaves: 
$$\IC_\lambda = \mathbb{F}_p [\dim \Gr_G^{\leq \lambda}] \in D_c^b(\Gr_G^{\leq \lambda}, \bbF_p).$$ Let $w_0$ be the longest element of the Weyl group of $(G,T)$. In what follows we will use a letter $L$ as a subscript or superscript to denote the corresponding objects for $L$. 

The connected components of $\Gr_P$ and $\Gr_L$ are in bijection via the map $q$. If $c \in \pi_0(\Gr_L)$ corresponds to $\Gr_L^c$ then we denote the corresponding reduced connected component of $\Gr_P$ by
$S_c$. By restricting $\CT_L^G$ to $S_c$ we get a decomposition by \emph{weight functors}: 
%cf. \ref{CTweight}:
$$\CT_L^G \cong \bigoplus_{c \in \pi_0(\Gr_L)} F_c.$$ 
Then the fact that $\CT_L^G$ preserves perversity is a consequence of the following theorem, which is unique to $\bbF_p$-sheaves, cf. \ref{CTperv}.

\begin{thm} \label{thm2}
%[\ref{CTperv}] 
For $\lambda \in X_*(T)^+$ we have
$$
F_c(\IC_\lambda)=
\left\{ \begin{array}{ll} 
\IC^L_{w_0^Lw_0(\lambda)} & \quad \textrm{if $w_0(\lambda)(t) \in \Gr_L^c$}, \\
0 & \quad \textrm{otherwise}.
\end{array}
\right.
$$
\end{thm}

Equivalently, Theorem \ref{thm2} computes the relative $\bbF_p$-cohomology with compact support of the so-called \emph{Mirkovi\'c-Vilonen cycles} for the Levi $L$. The proof relies on the dynamics of $\bbG_m$-schemes of Bialynicki-Birula and Drinfeld, together with the existence of $\bbF_p$-acyclic 
$\bbG_m$-equivariant resolutions of singularities of $\Gr_G^{\leq \lambda}$; see \ref{proofstrat} below for more details.

\medskip

Let us now comment on the tensor property of the functor $\CT_L^G$. The general strategy of proof is similar to the one of Baumann-Riche for 
$\overline{\mathbb{Q}}_\ell$-coefficients \cite[\S 15]{BR18}. It involves the Beilinson-Drinfeld global convolution Grassmannian, cf. \ref{recallconv}, and the key step is to show that a certain complex of sheaves is a perverse intermediate extension, cf. \ref{conviso}. We achieve it by appealing to the main results regarding perverse $\bbF_p$-sheaves on $F$-rational varieties \cite[1.6-1.7]{modpGr}. In contrast, the analogue of the ingredient used for 
$\overline{\mathbb{Q}}_\ell$-coefficients fails;  see \ref{obstruc} below.

\subsection{Tannakian interpretation} \label{introtann}
By \cite{modpGr} the functor of tensor endomorphisms of the fiber functor 
$\oplus_i R^i\Gamma:(P_{L^+G}(\Gr_G,\bbF_p),*) \rightarrow (\Vect_{\bbF_p},\otimes) $
is represented by an affine monoid scheme $M_G$ over $\bbF_p$. Via the Tannakian formalism this results in an equivalence
$$(P_{L^+G}(\Gr_G, \bbF_p), *) \cong (\Rep_{\bbF_p}(M_G), \otimes).$$
This construction is analogous to the \emph{geometric Satake equivalence} \cite{GeometricSatake}. The monoid $M_G$ is pro-solvable, but beyond this little is known. We will apply the functor $\CT_L^G$ to deduce more information about $M_G$.

By Theorem \ref{thm2}, the functor $\CT_L^G$ takes values in the symmetric monoidal subcategory $$P_{L^+L}(\Gr_{L, w_0^LX_*(T)_-}, \bbF_p) \subset P_{L^+L}(\Gr_L, \bbF_p)$$ associated to the submonoid $w_0^LX_*(T)_- \subset X_*(T)_{+/L}$ in the sense of \ref{notaGrGM}, and by \ref{transivityres} it intertwines the fiber functors. Thus denoting by $M_{L,w_0^LX_*(T)_-}$ the Tannakian monoid of $P_{L^+L}(\Gr_{L, w_0^LX_*(T)_-}, \bbF_p)$, the Tannaka dual to $\CT_L^G$ is a morphism of $\bbF_p$-monoid schemes $M_L \rightarrow M_G$ which factors as
\begin{equation} \label{MLthm} \xymatrix{
M_L \ar[r] & M_{L,w_0^LX_*(T)_-} \ar[r] & M_G.
}
\end{equation} 

We currently have a limited understanding of the morphisms in (\ref{MLthm}). This is related to our lack of information on the structure of the $\Ext$ groups in the corresponding categories of representations. However, if $L=T$ then we can say more. In this case, the category $P_{L^+T}(\Gr_T, \bbF_p)$ is semi-simple,
$$M_T = \Spec( \bbF_p[X_*(T)]),\quad M_{T,X_*(T)_-} = \Spec(\bbF_p[X_*(T)_-]),$$
and the following holds, cf. \ref{propsoffw}:

\begin{thm}\label{MTthm}
%[\ref{propsoffw}] 
The Tannaka dual of $\CT_T^G$ induces a morphism of monoids $M_T \rightarrow M_G$ which factors as an open immersion followed by a closed immersion:
$$\xymatrix{
M_T \ar[r] & M_{T,X_*(T)_-} \ar[r] & M_G.
}$$
\end{thm}

Note that $M_T$ is the torus over $\bbF_p$ with root datum dual to that of $T$. Thus, the morphism $M_T \rightarrow M_G$ in Theorem \ref{MTthm} is analogous to the reconstruction of the dual maximal torus in the dual group of $G$ in \cite{GeometricSatake}.

\medskip

There is another perspective on the morphism $M_{T,X_*(T)_-} \rightarrow M_G$ in Theorem \ref{MTthm} as follows. By \cite[1.2]{modpGr}, the subcategory of semi-simple objects
$$P_{L^+G}(\Gr_G, \bbF_p)^{\sss} \subset P_{L^+L}(\Gr_L, \bbF_p)$$ is a symmetric monoidal subcategory. Then the Tannakian monoid $M_G^{\sss}$ of $P_{L^+G}(\Gr_G, \bbF_p)^{\sss}$ identifies canonically with $M_{T,X_*(T)_-}$ by \ref{sssnota}. 

\begin{defn}
The Tannaka dual of the above inclusion of semi-simple objects is called the \emph{eigenvalues homomorphism}
$$\xymatrix{\pi_G \colon M_G \ar[r] & M_G^{\sss}.}$$ 
The morphism 
$$\xymatrix{ \fw \colon M_G^{\sss} \ar[r] & M_G}$$
equal to $M_{T,X_*(T)_-} \rightarrow M_G$ in Theorem \ref{MTthm} under the canonical identification $M_G^{\sss} = M_{T,X_*(T)_-}$ is called the \emph{weight section}.
\end{defn}

By construction these morphisms satisfy  
$$\pi_G \circ \fw = \id_{M_G^{\sss}}.$$ The Tannaka dual of the weight section can be viewed as a \emph{semi-simplification functor} $(P_{L^+G}(\Gr_G,\bbF_p),*) \rightarrow (P_{L^+G}(\Gr_G,\bbF_p),*)^{\sss}$. We refer to \ref{Semi-simplification} for more discussion on this perspective.

\subsection{Relation to mod $p$ Hecke algebras}
In this subsection alone we view $G$ as a split connected reductive group over $\bbF_q$. We assume that all relevant subgroups are also defined over $\bbF_q$. Let $E = \bbF_q(\!(t)\!)$ and $\mathcal{O} = \bbF_q[\![t]\!]$, and consider the  unramified mod $p$ Hecke algebra $$\mathcal{H}_G := \{f \colon G(E) \rightarrow \mathbb{F}_p \: | \: f \text{ has compact support and is } G(\mathcal{O}) \text{ bi-invariant}\}.$$ A basis for $\mathcal{H}_G$ is $\{\mathds{1}_\lambda\}_{\lambda \in X_*(T)^+}$ 
where $\mathds{1}_{\lambda}$ is the characteristic function of the double coset $G(\mathcal{O}) \lambda(t) G(\mathcal{O})$. 

Let $U_P$ be the unipotent radical of the parabolic subgroup $P$. Herzig \cite[\S2.3]{H11a} defined the mod $p$ Satake transform
$$\mathcal{S}_L^G \colon \mathcal{H}_G \rightarrow \mathcal{H}_L, \quad f \mapsto \left(g \mapsto \sum_{U_P(\mathcal{O}) \backslash U_P(E)} f(ug) \right).$$

As ind-schemes over $\bbF_q$, for $c \in \pi_0(\Gr_L)$ we have
$$S_c = (LU_P \cdot \Gr_L^c)_{\red} \subset \Gr_G.$$ Since $\Gr_G(\bbF_q) = G(E)/G(\mathcal{O})$ and $LU_P(\bbF_q) = U_P(E)$ then the function-sheaf dictionary sends $\CT_L^G$ to $\mathcal{S}_L^G$, cf. \cite[\S 4]{Centralmodp}. In contrast, for $\overline{\mathbb{Q}}_\ell$-coefficients the two transforms differ by the modulus character of $P$. The isomorphisms in Theorem \ref{thm2} hold over $\bbF_q$, so by using that the $\IC$-sheaves are constant we obtain a geometric proof of the following result due to Herzig.

\begin{cor}[{\cite[Prop. 5.1]{H11a}}] \label{Hecke1}
We have
$$\mathcal{S}_L^G \left( \sum_{\mu \, \leq_G \, \lambda} \mathds{1}_\mu \right)= \sum_{\mu \, \leq_L \, w_0^Lw_0(\lambda)} \mathds{1}_\mu.$$
\end{cor}

Note that $\mathcal{H}_T = \bbF_p[X_*(T)]$ where the characteristic function of $\nu(t) T(\mathcal{O})$ corresponds to $e^{\nu} \in \bbF_p[X_*(T)]$ for $\nu \in X_*(T)$. By taking $L=T$ we obtain:

\begin{cor} \label{Hecke2}
The mod $p$ Satake transform induces an isomorphism
\begin{eqnarray*}
\mathcal{S}_T^G \colon \mathcal{H}_G & \stackrel{\sim} {\longrightarrow}  & \bbF_p[X_*(T)_-] \\
\sum_{\mu \leq \lambda} \mathds{1}_{\mu}  & \lmapsto & e^{w_0(\lambda)}.
\end{eqnarray*}
\end{cor}

Note that Corollaries \ref{Hecke1} and \ref{Hecke2} are ultimately statements about counting $\bbF_q$-points mod $p$ on the Mirkovi\'c-Vilonen cycles. From this point of view, the resolutions of singularities which go into the proof of Theorem \ref{thm2} allow us to reduce this point counting to one on affine spaces.

\begin{rmrk}
In \cite[4.5]{Centralmodp} a particular isomorphism $\phi \colon \mathcal{H}_G \cong \bbF_p[X_*(T)_-]$ is constructed using the function-sheaf dictionary and the formula \cite[1.2]{modpGr} for the convolution product in $P_{L^+G}(\Gr_G, \bbF_p)$. Herzig's explicit formula \cite[Prop. 5.1]{H11a} is then used to check that $\phi = \mathcal{S}_T^G$. Here Theorem \ref{thm2} gives a purely geometric proof of the fact that $\phi = \mathcal{S}_T^G$.
\end{rmrk}

\subsection{Relation to mod $p$ Satake parameters}
As a consequence of Corollary \ref{Hecke2}, the $\bbF_p$-algebra $\mathcal{H}_G$ is commutative and the corresponding affine $\bbF_p$-scheme is identified with the \emph{space of Satake parameters}
$$
\sP:=\Spec( \bbF_p[X_*(T)_-]).
$$
From the geometric theory \ref{introtann}, this is the underlying scheme of the semi-simple monoid $M_G^{\sss}$. Now for each standard Levi $L$ as above, the functor $\CT^G_L$ preserves the subcategories of semi-simple objects by Theorem \ref{thm2}, hence by
%Tannaka
duality the morphism (\ref{MLthm}) admits a semi-simplification $M_L^{\sss}\ra M_G^{\sss}$. Then we have the following, cf. \ref{lemdecomp}, \ref{herzigstratif}.

\begin{thm} \label{MLss}
The morphism 
$$
\xymatrix{ M_L^{\sss}=\sP_L \ar[r] & M_G^{\sss}}=\sP
$$ 
defined by the constant term functor $\CT_L^G$ is an open immersion.  

Moreover, denoting by $\cL$ the finite set of standard Levi subgroups $T\subset L\subset G$ and setting
$$
\forall L\in\cL,\quad S_L:=\sP_L\setminus\bigcup_{\substack{ L'\in \cL\\ L'\subsetneq L}} \sP_{L'}\textrm{ equipped with its reduced structure},
$$
the space of Langlands parameters $\sP$ is stratified as:
$$
\sP=\bigcup_{L\in\cL} S_L.
$$
The stratum $S_L$ is isomorphic to $(\bbA^1\setminus\{0\})^{\rank \pi_0(\Gr_L)}$ and the closure relation among the strata is given by $\overline{S_L}=\cup_{L'\supset L} S_{L'}$.
\end{thm}
The underlying decomposition of the set $\sP(\overline{\bbF}_p)$ was originally defined by Herzig in \cite[\S 1.5, \S 2.4]{H11b}. The construction above makes the link with the Satake category $P_{L^+G}(\Gr_G,\bbF_p)$. 

\subsection{Obstructions to adapting proofs for $\overline{\mathbb{Q}}_\ell$-coefficients}\label{obstruc}
Let us now explain why the known proofs that $\CT_{L,\ell}^G$ preserves perversity and is a tensor functor fail for $\mathbb{F}_p$-sheaves. So that we can deal with $\bbF_p$ and $\overline{\mathbb{Q}}_\ell$-coefficients simultaneously let us set $\IC_{\lambda,\ell}$ to be the $\ell$-adic intersection cohomology sheaf of $\Gr_G^{\leq \lambda}$. Then $\IC_{\lambda, \blacktriangle}$ is either an $\bbF_p$-sheaf or a $\overline{\mathbb{Q}}_\ell$-sheaf depending on the value of $\blacktriangle \in \{\emptyset, \ell\}$.

For both $\overline{\mathbb{Q}}_\ell$-sheaves and $\bbF_p$-sheaves, there is a homological argument which reduces us to the case $L=T$. Then $\pi_0(\Gr_B) = X_*(T)$ and $(\Gr_T)_{\red}$ is a disjoint union of points indexed by $X_*(T)$, so that the weight functors are $$F_\nu = R\Gamma_c(S_\nu, \cdot)[2 \rho(\nu)], \quad \nu \in X_*(T),$$ where $\rho$ is half the sum of the positive roots. The fact that $F_\nu$ preserves perversity is equivalent to the statement that \begin{equation} \label{NeededVanishing}
H^i_c(S_\nu, \IC_{\lambda, \blacktriangle}) \neq 0 \: \Rightarrow \: i = 2\rho(\nu).\end{equation} 

By dimension estimates we have $H^i_c(S_\nu, \IC_{\lambda, \blacktriangle}) =0 $ if $i > 2\rho(\nu).$ For the other inequality, one observes that there is a $\mathbb{G}_m$-action on $\Gr_G$ such that $S_\nu(k)$ is the set of $k$-points of the $\nu$-component of the \emph{attractor} in the sense of \ref{D}. Then Braden's hyperbolic localization theorem \cite{BradenLoc} provides a comparison with the cohomology supported in the 
$\nu$-component of the \emph{repeller} (i.e. the attractor for the opposite $\bbG_m$-action), which leads to the other half of the desired vanishing (\ref{NeededVanishing}) for $\overline{\mathbb{Q}}_\ell$-coefficients, cf. \cite[Th. 3.5]{GeometricSatake}. However, we show in Appendix \ref{appA} that Braden's hyperbolic localization theorem \emph{fails} for $\bbF_p$-sheaves. Braden's theorem is also the key tool from the proof of the compatibility of $\CT_{L,\ell}^G$ with convolution \cite[1.15.2]{BR18}  that we lack in the case of $\bbF_p$-coefficients.

\medskip

There is another approach to proving (\ref{NeededVanishing}) due to Ng\^{o}-Polo \cite{ngopolo}. Let $\mathcal{M} \subset X_*(T)^+$ be the subset of cocharacters that are either minuscule or quasi-minuscule. If $\lambda$ is quasi-minuscule then Ng\^{o}-Polo construct a resolution of $\Gr_G^{\leq \lambda}$ and explicitly stratify the fiber over $S_\nu \cap \Gr_G^{\leq \lambda}$ by affine spaces. These stratifications allow one to estimate the dimension of $H^i_c(S_\nu, \IC_{\lambda, \blacktriangle})$ for $(\nu, \lambda) \in  X_*(T) \times \mathcal{M}$.

If $\lambda\in X_*(T)^+$ can be decomposed as a sum of elements of $\mathcal{M}$, then by considering the corresponding convolution Grassmannian $m \colon \Gr_{G}^{\leq \lambda_\bullet} \rightarrow \Gr_G^{\leq \lambda}$ the previous estimates allow one to prove (\ref{NeededVanishing}) for any direct summand of $Rm_!(\IC_{\lambda_\bullet, \blacktriangle})$, where $\IC_{\lambda_\bullet, \blacktriangle}$ is the IC-sheaf of $\Gr_{G}^{\leq \lambda_\bullet}$. This is sufficient to complete the argument for $\overline{\mathbb{Q}}_\ell$-sheaves. However, for $\bbF_p$-sheaves we have $Rm_!(\IC_{\lambda_\bullet}) = \IC_\lambda$ by \cite[6.5]{modpGr}. Thus in our situation Ng\^{o}-Polo's approach allows us to conclude for groups of type $A_n$ only, since this is the only case where the fundamental coweights freely generating $X_*(T_{\ad})^+$ belong to the subset $\mathcal{M}_{\ad} \subset X_*(T_{\ad})^+$.
% (and then there are even minuscule) 

\subsection{Proof strategy for preservation of perversity.} \label{proofstrat}
Our approach to proving Theorem \ref{thm2} combines ideas from both \cite{GeometricSatake} and \cite{ngopolo}, and works directly for $L$ not necessarily equal to $T$. We start with the observation that there is a $\mathbb{G}_m$-action on $\Gr_G$ such that $\Gr_L(k) = \Gr_G(k)^{\mathbb{G}_m(k)}$ and such that the $S_c(k)$ for $c \in \pi_0(\Gr_L)$ are the sets of $k$-points of the components of the attractor:
$$
\forall c\in \pi_0(\Gr_L),\quad S_c(k) = \{x\in \Gr_G(k)\ |\ \lim_{k^{\times}\ni z\ra 0} z\cdot x \in \Gr_L^c(k)\}.
$$ 
Then the (unshifted) weight functor $F_c$ identifies with the hyperbolic localization functor of relative cohomology with compact support flowing in the direction of the fixed points $\Gr_L^c$.

Let $\mathcal{B}$ be the Iwahori group scheme equal to the dilation of $G_{k[\![t]\!]}$ along $B_k$. The affine flag variety $\mathcal{F}\ell: = LG/L^+\mathcal{B}$ is a $\mathbb{G}_m$-equivariant $G/B$-fibration over $\Gr_G$. Unlike the case of $\overline{\mathbb{Q}}_\ell$-coefficients, the flag variety $G/B$ is acyclic for $\bbF_p$-coefficients in the sense that $R\Gamma(G/B, \bbF_p) = \bbF_p[0].$ This allows us to compare $F_c(\IC_\lambda)$ with hyperbolic localizations on the preimage of $S_c \cap \Gr_G^{\leq \lambda}$ in $\mathcal{F}\ell$. 

Next we note that any Schubert variety in $\mathcal{F}\ell$ admits a so-called Demazure resolution, which is both $\mathbb{G}_m$-equivariant and 
$\bbF_p$-acyclic. 
%cf. \ref{resolD}. 

Then we can appeal to a general result of Bialynicki-Birula on the structure of smooth proper $\mathbb{G}_m$-varieties: 
%cf. \ref{BB}
on the resolution, there is a unique closed attractor component, while the other components are positive-dimensional affine bundles over their fixed points. Such bundles have no relative $\bbF_p$-cohomology with compact support, 
%cf. \ref{affinecoh} and proper base change, 
so only the closed component contributes. 

The final complete determination of $F_c(\IC_\lambda)$ relies on the affineness of Drinfeld's attractor of a not necessarily smooth $\bbG_m$-scheme.

\subsection{Outline} In Section \ref{sec_Gm} we recall results of Bialynicki-Birula and Drinfeld on the structure of schemes with a $\mathbb{G}_m$-action. The main result is \ref{vanishstrategy} on $\bbF_p$-cohomology with compact support in the attractors on a general class of $\mathbb{G}_m$-schemes. In Section \ref{section_MVcycle} we apply this result on the affine Grassmannian to prove \ref{thmcompute}, which is the main input in the proof of Theorem \ref{thm2}. In Sections \ref{absolutesection} and \ref{totalsection} we prove Theorems \ref{thm1} and \ref{thm2} in the case $L=T$. We treat the case of general $L$ in Section \ref{section_reslevi}. In Section \ref{sectionTannakian} we investigate the Tannakian consequences of Theorems \ref{thm1} and \ref{thm2} for the monoid $M_G$. In Section \ref{section_space} we study the stratification of $\sP$ induced by the morphisms $M_L^{\sss} \rightarrow M_G^{\sss}$. Finally, in Appendix \ref{appA} we show that Braden's hyperbolic localization theorem is false for $\bbF_p$-coefficients. 

\bigskip

\textbf{Notation.}
Let $k$ be an algebraically closed field of characteristic $p > 0$ and let $G$ be a connected reductive group over $k$. Fix a maximal torus and a Borel subgroup $T \subset B \subset G$, and let $U \subset B$ be the unipotent radical of $B$. Let $W$ be the Weyl group of $G$ and let $w_0 \in W$ be the longest element. 

Let $X^*(T)$ and $X_*(T)$ be the lattices of characters and cocharacters of $T$, and $X_*(T)^+$ (resp. $X_*(T)_-$) the monoid of dominant (resp. antidominant) cocharacters determined by $B$. Let $\Phi$ and $\Phi^\vee$ be the sets of roots and coroots,  $\Phi^+$ and $(\Phi^+)^\vee$ the subsets of positive roots and positive coroots, and $\Delta$ and $\Delta^\vee$ the subsets of simple roots and simple coroots. For $\nu$, $\nu' \in X_*(T)$ we write $\nu \leq \nu'$ if $\nu' - \nu$ is a sum of positive coroots with non-negative integer coefficients. Let $\rho$ and $\hat{\rho}$ be respectively half the sum of the positive roots and coroots. For $\nu \in X_*(T)$ let $\rho(\nu) \in \mathbb{Z}$ be the pairing of $\rho$ and $\nu$.

\bigskip

\textbf{Acknowledgments.} During the preparation of this article R.C. was partially supported by the National Science Foundation Graduate Research Fellowship Program under Grant No. DGE-1144152, and C.P. was partially supported by the Agence Nationale pour la Recherche COLOSS project ANR-19-CE40-0015.

\section{Some general computations of $\bbF_p$-cohomology with compact support} \label{sec_Gm}
\numberwithin{defn}{subsection}
\subsection{The affine space}

\begin{lem} \label{affinecoh}
Let $\mathbb{A}^d$ be the affine space over $k$ of dimension $d$. Then
$$
R\Gamma_c(\mathbb{A}^d, \mathbb{F}_p) = 
\left\{ \begin{array}{ll} 
\bbF_p[0] & \quad \textrm{if $d=0$}, \\
0 & \quad \textrm{otherwise}.
\end{array}
\right.
$$
\end{lem}

\begin{proof}
We can assume $d>0$. Consider the open immersion $j \colon \mathbb{A}^d \rightarrow \mathbb{P}^d$ and the complementary closed immersion $i \colon \bbP^{d-1} \rightarrow \bbP^d$. This gives rise to an exact triangle
$$
\xymatrix{
Rj_! \mathbb{F}_p[0] \ar[r] & \mathbb{F}_p[0] \ar[r] & Ri_* \mathbb{F}_p[0] \ar[r]^<<<<<{+1} &.
}
$$
From \cite{S55}, we know that $$\forall i >0, \quad H^i(\mathbb{P}^d, \mathcal{O}_{\mathbb{P}^d}) = H^i(\mathbb{P}^{d-1}, \mathcal{O}_{\mathbb{P}^{d-1}}) = 0.$$ Thus since $H^0(\mathbb{P}^d, \mathcal{O}_{\mathbb{P}^d}) = H^0(\mathbb{P}^{d-1}, \mathcal{O}_{\mathbb{P}^{d-1}}) = k$, then by the Artin-Schreier sequence the map $R\Gamma(\mathbb{F}_p[0]) \rightarrow R\Gamma(Ri_* \mathbb{F}_p[0])$ is a quasi-isomorphism. Hence $R\Gamma(Rj_! \mathbb{F}_p[0]) = 0$, i.e. $R\Gamma_c(\mathbb{A}^d,\mathbb{F}_p)=0$.  
\end{proof} 

\subsection{Schemes admitting a decomposition by affine spaces}

\begin{nota}
Given a scheme $X$, we denote by $|X|$ its underlying topological space.
\end{nota}

\begin{defn} \label{defdecompfil}
Let $X$ be a scheme. 
\begin{itemize}
\item A \emph{decomposition} of $X$ is a family of 
%locally closed 
subschemes $X_i\subset X$, $i\in I$, such that 
$$
|X|=\bigcup_{i\in I} |X_i|\quad\textrm{and}\quad\textrm{$|X_i|\cap |X_j|=\emptyset$ for all $i\neq j$}.
$$
\item A \emph{filtration} of $X$ is a finite decreasing sequence of closed subschemes
$$
X=Z_0\supset Z_1\supset\cdots\supset Z_{N-1}\supset Z_N=\emptyset.
$$
The 
%locally closed 
subschemes $Z_n\setminus Z_{n+1}$, $n=0,\ldots,N-1$, are the \emph{cells} of the filtration.
\end{itemize}
\end{defn}

\begin{cor} \label{vanishfilaff}
Let $X$ be a $k$-scheme. Assume that $X$ admits a filtration whose cells are positive dimensional affine spaces. Then
$$
R\Gamma_c(X, \mathbb{F}_p) = 0.
$$
\end{cor}

\begin{proof}
This follows from \ref{affinecoh} and the long exact sequence of $\bbF_p$-cohomology with compact support associated to the decomposition of a scheme into an open and a complementary closed subscheme.
\end{proof}

\subsection{Some $\bbG_m$-schemes} \label{Gmsec}
Let $X$ be a scheme of finite type over $k$, equipped with a  $\bbG_m$-action. Recall from \cite{D13} the following definitions and results.

\begin{defn} \label{D}
\begin{itemize}
\item The \emph{space of fixed points} is the fppf sheaf
$$
X^0:=\uHom^{\bbG_m}_k(\Spec(k),X)
$$
where $\Spec(k)$ is equipped with the trivial $\bbG_m$-action.
\item The \emph{attractor} is the fppf sheaf
$$
X^+:=\uHom^{\bbG_m}_k((\bbA^1)^+,X)
$$
where $(\bbA^1)^+$ is the affine line over $k$ equipped with the $\bbG_m$-action by dilations.
\end{itemize}
\end{defn}

Evaluating at $1$ and $0$ defines maps $p$ and $q$:
$$
\xymatrix{
&X^+ \ar[dl]_q \ar[dr]^p & \\
X^0 & & X.
}
$$
The space of fixed points is representable by a closed subscheme $X^0\subset X$. The attractor is representable by a $k$-scheme. The morphism $q$ is affine, and the section $X^0\subset X^+$ obtained by precomposing with the structural morphism $(\bbA^1)^+\ra\Spec(k)$ induces an identification
$(X^+)^0=X^0$; the morphism $p$ restricts to the identity between $X^0\subset X^+$ and $X^0\subset X$. Moreover, the morphism $q$ has geometrically connected fibers, cf. \cite[Cor. 1.12]{R19}, so that the decomposition of $X^+$  as a disjoint union of its connected components is the preimage by $q$ of the corresponding decomposition of $X^0$:
$$
X^+=\coprod_{i\in\pi_0(X^0)}X_i.
$$
For $i\in \pi_0(X^0)$ we will denote by $q_i:X_i\ra X_i^0$ the induced retraction.

\begin{rmrk}\label{kpoints}
Suppose that $X$ is separated over $k$. Then $p:X^+\ra X$ is a monomorphism, which induces the following identifications of sets:
$$
X^+(k)\simeq\{x\in X(k)\ |\ \textrm{$\lim_{k^{\times}\ni z\ra 0} z\cdot x$ exists} \},
$$
\begin{eqnarray*}
q(k):X^+(k)&\lra& X^0(k)\\
x & \lmapsto & \lim_{k^{\times}\ni z\ra 0} z\cdot x,
\end{eqnarray*}
and for each $i\in \pi_0(X^0)$, 
$$
X_i(k)\simeq\{x\in X(k)\ |\ \lim_{k^{\times}\ni z\ra 0} z\cdot x\in X^0_i(k)\}.
$$
\end{rmrk}

Now consider the following hypothesis:

\medskip

(H) \emph{for each $i\in\pi_0(X^0)$, the restriction $p|_{X_i}:X_i\ra X$ is an \emph{immersion}.}

\begin{lem}\label{decompfil}
\begin{enumerate}
\item
Suppose that (H) is satisfied, and that $X$ is proper over $k$. Then the family of subschemes $(X_i)_{i\in \pi_0(X^0)}$ is a decomposition of $X$. 
\item
Suppose that there exists a $\bbG_m$-equivariant immersion of $X$ into some projective space $\bbP(V)$ where $\bbG_m$ acts linearly on $V$. Then (H) is satisfied, and if moreover $X$ is proper, there exists a filtration $(Z_n)_{0\leq n\leq |\pi_0(X^0)|}$ of $X$ having $(X_i)_{i\in \pi_0(X^0)}$ as its family of cells.
\end{enumerate}
\end{lem}

\begin{proof}
(1) When $X$ is proper over $k$, then $p$ is universally bijective by \cite[1.4.11 (iii)]{D13}. In particular 
$$
|X|=\bigcup_{i\in I} p(|X_i|)\quad\textrm{and}\quad\textrm{$p(|X_i|)\cap p(|X_j|)=\emptyset$ for all $i\neq j$}.
$$
When (H) is satisfied, then for each $i$ there exists a unique subscheme $p(X_i)\subset X$ such that $p|_{X_i}$ decomposes as an isomorphism $X_i\xrightarrow{\sim}p(X_i)$ followed by the canonical immersion $p(X_i)\subset X$. Thus, identifying $X_i$ with $p(X_i)$, we get that the family 
$(X_i)_{i\in \pi_0(X^0)}$ is a decomposition of $X$.

(2) When $X$ admits a $\bbG_m$-equivariant immersion into some projective space $\bbP(V)$ where $\bbG_m$-acts linearly on $V$, then, as noted in \cite[B.0.3 (iii)]{D13}, the fact that (H) is satisfied follows from the case $X=\bbP(V)$. If the immersion is closed, the fact that the decomposition $(X_i)_{i\in \pi_0(X^0)}$ of $X$ can be realized as the cells of a filtration follows again from the case $X=\bbP(V)$, as proved in \cite[Th. 3]{BB76}\footnote{As noted in the remark following the proof of the theorem in loc. cit., the smoothness assumption on the closed $\bbG_m$-subscheme $X\subset \bbP(V)$ is not used in that proof. The existence of such a filtration is also  recorded in \cite[Lem. 4.12]{Car02}.}.
\end{proof}

\begin{thm}\label{BB}
\begin{enumerate}
\item Suppose that $X$ is smooth and separated over $k$. Then (H) is satisfied, $X^0$ and $X^+$
%equivalently each X_i
are smooth over $k$, 
and for each $i \in \pi_0(X^0)$,  there exists an integer $d_i\geq 0$ such that
$$
\xymatrix{
X_i \ar[rr]^{\sim} \ar[dr]_{q_i} && \bbA^{d_i}\times X^0_i \ar[dl]^{\pr_2} \\
&X^0_i&
}
$$ 
Zariski-locally on $X^0_i$. If moreover $X$ is proper over $k$, then 
%$d_i=0$ if and only if 
$X_i\subset X$ is closed if and only if $X_i=X_i^0$, and there exists exactly one such $X_i$ lying in each connected component of $X$.
\item Suppose that $X$ is normal and projective over $k$. Then there exists a $\bbG_m$-equivariant closed immersion of $X$ into some projective space $\bbP(V)$ where $\bbG_m$-acts linearly on $V$.
\end{enumerate}
\end{thm}

\begin{proof}
(1) The scheme $X^0$ is smooth over $k$ by \cite[Prop. 4]{F71}. The other results are contained in \cite{BB73}.

(2) This is a result of \cite{S74}.
\end{proof}

\begin{cor}\label{vanishstrategy}
Let $X$ be a proper $k$-scheme equipped with a $\bbG_m$-action satisfying (H). Suppose that there exists a connected smooth projective $k$-scheme $\widetilde{X}$ equipped with a $\bbG_m$-action, and a surjective 
%actually here it would the enough to require the surjectivity of the morphism induced by $\pi$ between the $\pi_0$ of the schemes of $\bbG_m$-fixed points 
$\bbG_m$-equivariant morphism of $k$-schemes
$$
f:\widetilde{X}\lra X.
$$
Then there exists at most one $i=:i_0\in \pi_0(X^0)$ such that $X_i\subset X$ is closed.

Suppose moreover that $Rf_*\bbF_p=\bbF_p[0]$. 
%here remark that it implies the surjectivity of f
Then for $i\in \pi_0(X^0)$, we have:
$$
R(q_i)_!\mathbb{F}_p= 
\left\{ \begin{array}{ll} 
\mathbb{F}_p|_{X_{i_0}^0}[0] & \quad \textrm{if $i=i_0$}, \\
0 & \quad \textrm{otherwise}.
\end{array}
\right.
$$
\end{cor}

\begin{rmrk}\label{vanishstrategy+}
If $X$ can be embedded equivariantly into some $\bbP(V)$ where $\bbG_m$ acts linearly on $V$, then by \ref{decompfil} (2) there exists at least one  $i\in \pi_0(X^0)$ such that $X_i\subset X$ is closed, hence then there is exactly one such $i$.
\end{rmrk}

\begin{proof}[Proof of Corollary \ref{vanishstrategy}]
Let $i\in \pi_0(X^0)$. Define $Y_i$ and $f_i$ by the fiber product diagram
$$
\xymatrix{
Y_i \ar[r]^{f_i} \ar[d] & X_i \ar[d]^{p|_{X_i}} \\
\widetilde{X}\ar[r]^{f} & X.
}
$$
%f^{-1}(X_i)
Since $p|_{X_i}$ is an immersion by hypothesis, so is the canonical map $Y_i\ra \widetilde{X}$, and we write $X_i\subset X$ and $Y_i\subset \widetilde{X}$ for the corresponding subschemes.
Also by \ref{BB} (1) the schemes $\widetilde{X}_j$, $j\in \pi_0(\widetilde{X}^0)$, are realized as subschemes of $\widetilde{X}$, and they form a decomposition of the latter, cf. \ref{decompfil} (1). Then we have the following identity of subspaces of $|\widetilde{X}|$:
$$
|Y_i|=\bigcup_{\substack{j\in \pi_0(\widetilde{X}^0)\\ f(j)=i}}|\widetilde{X}_j|\ ;
$$
indeed this can be checked on $k$-points, where it follows from the definitions, cf. \ref{kpoints}. Thus the immersions $\widetilde{X}_j\ra\widetilde{X}$, for $f(j)=i$, factor through $Y_i\subset \widetilde{X}$ (note that the schemes $\widetilde{X}_j$ are reduced, cf. \ref{BB} (1)), and the family $(\widetilde{X}_j)_{f(j)=i}$ is a decomposition of the scheme $Y_i$. Further, by \ref{BB} (2) and \ref{decompfil} (2), one may form a filtration of $\widetilde{X}$,
$$
\widetilde{X}=Z_0\supset Z_1\supset\cdots\supset Z_{N-1}\supset Z_N=\emptyset, \quad N:=|\pi_0(\widetilde{X}^0)|,
$$
whose family of cells is $(\widetilde{X}_j)_{j\in \pi_0(\widetilde{X}^0)}$. Intersecting 
%in the scheme-theoretical sense
with $Y_i$ we get a filtration of $Y_i$
$$
Y_i=Z_{i,0}\supset Z_{i,1}\supset\cdots\supset Z_{i,N-1}\supset Z_{i,N}=\emptyset
$$
whose family of nonempty cells is $(\widetilde{X}_j)_{f(j)=i}$.

Now suppose that $X_i\subset X$ is closed. Then so is $Y_i\subset \widetilde{X}$. Moreover the assumption that $f$ is surjective ensures that $Y_i$ is nonempty. Hence, if $N_i$ is the greatest integer $n\leq N$ such that $Z_{i,n}$ is nonempty, then $Z_{i,N_i}$ is equal to some $\widetilde{X}_j$ with $f(j)=i$ which is closed in ($Y_i$ hence in) $\widetilde{X}$. But since $\widetilde{X}$ is connected, there is exactly one $\widetilde{X}_j\subset \widetilde{X}$ which is closed, say $\widetilde{X}_{j_0}$, by \ref{BB} (1). Thus $i=f(j_0)=:i_0$ is uniquely determined.

Finally, suppose moreover that $Rf_!\bbF_p=\bbF_p[0]$. If $i=i_0$, then $R(q_{i_0})_!\bbF_p=\bbF_p|_{X_{i_0}^0}[0]$ by \ref{trivred} below. If $i\neq i_0$, consider the commutative diagram
$$
\xymatrix{
Y_i\ar[r]^<<<<<{f_i}\ar[dr]_{q_{Y_i}:=} & X_i \ar[d]^{q_i} \\
& X_i^0. 
}
$$
By proper base change $R(f_i)_!\bbF_p=\bbF_p[0]$ and 
$$
R(q_i)_!\bbF_p=R(q_{Y_i})_!\bbF_p.
$$
Then recall the filtration of $Y_i$ constructed above. For every $0\leq n\leq N-1$ such that $Z_{i,n}\setminus  Z_{i,n+1}$ is nonempty, let $i_n:Z_{i,n+1}\ra Z_{i,n}$ be the corresponding closed immersion, $j_n:\widetilde{X}_n\ra Z_{i,n}$ be the complementary open immersion, and in $D_c^b(Z_{i,n},\bbF_p)$ form  the exact triangle
$$
\xymatrix{
Rj_{n!} \mathbb{F}_p[0] \ar[r] & \mathbb{F}_p[0] \ar[r] & Ri_{n*} \mathbb{F}_p[0] \ar[r]^<<<<<{+1} &.
}
$$ 
Setting $q_{Z_{i,n}}:=q_{Y_i}|_{Z_{i,n}}:Z_{i,n}\ra X_i^0$ and applying $R(q_{Z_{i,n}})_!$ we get the exact triangle
$$
\xymatrix{
R(q_{Z_{i,n}}\circ j_n)_! \mathbb{F}_p[0] \ar[r] & R(q_{Z_{i,n}})_!\mathbb{F}_p \ar[r] & R(q_{Z_{i,n+1}})_! \mathbb{F}_p \ar[r]^<<<<<{+1} &
}
$$ 
in $D_c^b(X_i^0,\bbF_p)$. By construction, the morphism $q_{Z_{i,n}}\circ j_n:\widetilde{X}_n\ra X^0$ is equal to $q_i\circ (f_i|_{\widetilde{X}_n})$, and we have the commutative diagram
$$
\xymatrix{
\widetilde{X}_n \ar[d]_{q_n} \ar[r]^{f_i|_{\widetilde{X}_n}} & X_i \ar[d]^{q_i} \\
\widetilde{X}_n^0 \ar[r] & X_i^0
}
$$
functorially induced by $f$. Here $\widetilde{X}_n\neq \widetilde{X}_{j_0}$ since $i\neq i_0$. Consequently $R(q_n)_!\bbF_p=0$ by proper base change, \ref{BB} (1) and \ref{affinecoh}. Thus 
$$
\xymatrix{
 R(q_{Z_{i,n}})_!\mathbb{F}_p \ar[r]^<<<<<{\sim} & R(q_{Z_{i,n+1}})_! \mathbb{F}_p.
}
$$ 
Descending in this way along the filtration of $Y_i$, we obtain
$$
\xymatrix{
R(q_{Y_i})_!\bbF_p \ar[r]^<<<<<{\sim} & R(q_{\emptyset})_! \mathbb{F}_p=0,
}
$$
which concludes the proof.
\end{proof}

\begin{lem}\label{trivred}
Let $X$ be a proper $k$-scheme equipped with a $\bbG_m$-action satisfying (H). Then for each $i\in\pi_0(X^0)$ such that $X_i\subset X$ is closed, the retraction $q_i: X_i\ra X_i^0$ is a universal homeomorphism and the section $X_i^0\subset X_i$ induces the identity of reduced schemes 
$(X_i^0)_{\red}=(X_i)_{\red}$.
\end{lem}

\begin{proof}
As we have recalled, the retraction $q:X^+\ra X^0$ is always affine, \cite[Th. 1.4.2 (ii)]{D13}, with geometrically connected fibers, cf. \cite[Cor. 1.12]{R19}. In particular its restrictions $q_i:X_i\ra X_i^0$ above each $X_i^0$ have the same properties. 

Now let $i\in\pi_0(X^0)$ such that $X_i\subset X$ is closed. Then $X_i$ is proper over $k$, so that the morphism $q_i$ is proper. Consequently, in this case $q_i$ is a universal homeomorphism. So its canonical section $X_i^0\subset X_i$ identifies $(X_i^0)_{\red}$ and $(X_i)_{\red}$.
\end{proof}

\section{$\bbF_p$-cohomology with compact support of the MV-cycles} \label{section_MVcycle}

\subsection{The affine Grassmannian}

For an affine group scheme $H$ over $k$ (or more generally, over $k[\![t]\!]$) we have the loop group functor
$$LH \colon k\text{-Algebras} \rightarrow \text{Sets}, \quad \quad R \mapsto H(R(\!(t)\!)),$$ and the non-negative loop group functor
$$L^+H \colon k\text{-Algebras} \rightarrow \text{Sets}, \quad \quad R \mapsto H(R[\![t]\!]).$$
The affine Grassmannian of $G$ is the fpqc-quotient $\Gr_G:=LG/L^+G$. It is represented by an ind-scheme over $k$.

\subsection{The Cartan decomposition}
The set $X_*(T)^+$ embeds in $\Gr_G(k)$ via the identification $\lambda\mapsto \lambda(t)$. For $\lambda\in X^*(T)^+$, denote by $\Gr_G^{\lambda}$ the reduced $L^+G$-orbit of $\lambda(t)$ in $\Gr_G$. Then we have the decomposition of the reduced ind-closed subscheme $(\Gr_G)_{\red}\subset \Gr_G$:
$$
(\Gr_G)_{\red}=\bigcup_{\lambda\in X_*(T)^+}\Gr_G^{\lambda},
$$
which on $k$-points is the quotient of the \emph{Cartan decomposition} of $G(k(\!(t)\!))$:
$$
G(k(\!(t)\!))= \bigcup_{\lambda\in X_*(T)^+}G(k[\![t]\!])\lambda(t) G(k[\![t]\!]).
$$
Let $\overline{\Gr_G^{\lambda}}$ be the closure of $\Gr_G^{\lambda}$ in $\Gr_G$. Then $\overline{\Gr_G^{\lambda}}$ is an integral projective $k$-scheme, of dimension $2\rho(\lambda)$, which is the union of the $\Gr_G^{\mu}$ with $\mu\leq\lambda$; it will also be denoted by $\Gr_{G}^{\leq\lambda}$. Moreover $(\Gr_G)_{\red}$ is the limit of the $\overline{\Gr_G^{\lambda}}$:
$$
(\Gr_G)_{\red}=\varinjlim_{\lambda\in X_*(T)^+}\overline{\Gr_G^{\lambda}}.
$$

\subsection{The Iwasawa decomposition} \label{Iwasec}
From our fixed choice $B=U\rtimes T\subset G$, we have the quotient map $B\ra T$ and the closed immersion $B\ra G$:
$$
\xymatrix{
&B \ar[dl] \ar[dr] & \\
T & & G.
}
$$
Then by functoriality we get a diagram
$$
\xymatrix{
&\Gr_B \ar[dl] \ar[dr] & \\
\Gr_T & & \Gr_G.
}
$$
Passing to the reductions, we get the decomposition of $(\Gr_B)_{\red}$ into its connected components
$$
(\Gr_B)_{\red}= \coprod_{\nu\in X_*(T)}S_{\nu}
$$
and a decomposition of $(\Gr_G)_{\red}$ by ind-subschemes
$$
(\Gr_G)_{\red}= \bigcup_{\nu\in X_*(T)}S_{\nu},
$$
where $X_*(T)$ is embedded in $\Gr_G(k)$ via the identification $\nu\mapsto \nu(t)$. On $k$-points, it is the quotient of the \emph{Iwasawa decomposition} of $G(k(\!(t)\!))$:
$$
G(k(\!(t)\!))= \bigcup_{\nu\in X_*(T)}U(k(\!(t)\!))\nu(t) G(k[\![t]\!]).
$$
%For $\nu\in X_*(T)$ the closure $\overline{S_{\nu}}$ is the union of the $S_{\nu'}$ with $\nu'\leq \nu$.

\subsection{The Mirkovi\'{c}-Vilonen cycles} \label{MVsec}

\begin{defn}
Let $(\nu,\lambda)\in X_*(T)\times X_*(T)^+$. The MV-cycle of index $(\nu,\lambda)$ is the reduced $k$-scheme
$$
S_{\nu}\cap \overline{\Gr_G^{\lambda}}.
$$
\end{defn}

The MV-cycles can be reconstructed from the theory of $\bbG_m$-schemes, as follows. 

\medskip

The adjoint action of the torus $T$ on $LG$ normalizes $L^+G$ and hence induces an action on $\Gr_G$. Fixing a regular dominant cocharacter $\bbG_m\ra T$, we equip $\Gr_G$ with the resulting $\bbG_m$-action. 

Let $\lambda\in X_*(T)^+$. Then $\Gr_G^{\lambda}$ and $\overline{\Gr_G^{\lambda}}$ are stable under the $\bbG_m$-action. Thus $$X:=\overline{\Gr_G^{\lambda}}$$ is a projective $\bbG_m$-scheme over $k$. Moreover, it can be embedded equivariantly in some $\bbP(V)$ where 
$\bbG_m$ acts linearly on $V$: indeed, one can construct on the affine Grassmannian $\Gr_G$ some \emph{$G$-equivariant}
% hence some \emph{$\bbG_m$-equivariant} 
 very ample line bundle, cf. \cite[\S 1.5]{Z17}.  Consequently by \ref{decompfil} (2) the connected components of the attractor $X^+$ are realized as subschemes of $X$. Then, it follows from (\ref{kpoints} and) the Iwasawa decomposition of $G(k(\!(t)\!))$ that 
 $$
 X^0(k)=X_*(T)\cap X
 \quad\textrm{and}\quad
\forall \nu\in X^0(k),\ X_{\nu}(k)=(S_{\nu}\cap \overline{\Gr_G^{\lambda}})(k).
$$ 
Thus the MV-cycles indexed by $(\nu,\lambda)$ for varying $\nu$ are precisely the $(X_{\nu})_{\red}\subset X$, which decompose $X$ as
$$
X=\bigcup_{\nu\in X_*(T)\cap X} (X_{\nu})_{\red}.
$$

\subsection{Generalization to the standard Levi subgroups} \label{Levisec}
Let $P=U_P\rtimes L\subset G$ be a parabolic subgroup of $G$ containing $B$ with unipotent radical $U_P$ and Levi factor $L$. Then
$$
\xymatrix{
&P \ar[dl] \ar[dr] & \\
L & & G
}
$$
induces
$$
\xymatrix{
&\Gr_P \ar[dl] \ar[dr] & \\
\Gr_L & & \Gr_G,
}
$$
the decomposition of $(\Gr_P)_{\red}$ into its connected components
$$
(\Gr_P)_{\red}= \coprod_{c\in\pi_0(\Gr_L)}S_{c}
$$
and a decomposition of $(\Gr_G)_{\red}$ by ind-subschemes
$$
(\Gr_G)_{\red}= \bigcup_{c\in\pi_0(\Gr_L)}S_{c}.
$$
%indeed clearly injective and surjectivity deduced from the case $P=B$ using $\Gr_B(k)\ra \Gr_P(k)\ra \Gr_G(k)$, i.e. results from the Iwasawa decomposition

\begin{defn}
Let $(c,\lambda)\in \pi_0(\Gr_L)\times X_*(T)^+$. The MV-cycle of index $(c,\lambda)$ is the reduced $k$-scheme
$$
S_c\cap \overline{\Gr_G^{\lambda}}.
$$
\end{defn}

Fix a dominant cocharacter $\bbG_m\ra T$ whose centralizer in $G$ is equal to $L$, and equip $\Gr_G$ with the restriction to $\bbG_m$ of the adjoint action of $T$ along this cocharacter.

Let $\lambda\in X_*(T)^+$ and $X:=\overline{\Gr_G^{\lambda}}$. The connected components of the attractor $X^+$ are realized as subschemes of $X$, 
and $X^0(k)=(\Gr_L\cap X)(k)$.

\begin{lem}
Let $c\in \pi_0(\Gr_L)$. Then $\Gr_L^c\cap X$ is connected.
\end{lem}

\begin{proof}
Indeed $\Gr_L^c\cap X=\Gr_L^c\cap \overline{\Gr_G^{\lambda}}$ is a closed $L^+L$-stable subscheme of $\Gr_L^c$, hence a union of Cartan closures for the affine Grassmannian $\Gr_L$ which are contained in the connected component $\Gr_L^c$. Such Cartan closures are irreducible, and all contain the unique minimal $L^+L$-orbit of $\Gr_L^c$, so any union of them is connected.
\end{proof}
\noindent It follows that 
$$
\pi_0(X^0)=\{|\Gr_L^c\cap X|\ |\ \textrm{$c\in\pi_0(\Gr_L)$ and $\Gr_L^c\cap X\neq\emptyset$}\}.
$$
Next, the bijection $\Gr_P(k)\xrightarrow{\sim}\Gr_G(k)$ corresponds to the decomposition
$$
G(k(\!(t)\!))/G(k[\![t]\!])= \bigcup_{c\in \pi_0(\Gr_L)}S_c(k)=\bigcup_{c\in \pi_0(\Gr_L)}U_P(k(\!(t)\!))\Gr_L^c(k),
$$
and so we compute using \ref{kpoints} that
$$
\forall c\in \pi_0(X^0),\quad (X_c)_{\red}=S_c\cap \overline{\Gr_G^{\lambda}}.
$$
Thus the MV-cycles indexed by $(c,\lambda)$ for varying $c$ are precisely the $(X_c)_{\red}\subset X$,  and they decompose $X$ as
$$
X=\bigcup_{\substack{c\in \pi_0(\Gr_L)\\ \Gr_L^c\cap X\neq\emptyset}} (X_c)_{\red}.
$$

\subsection{Equivariant resolutions of Schubert varieties} 
Let 
$$
W_a:=\bbZ\Phi^{\vee}\rtimes W\subset \widetilde{W}:= X_*(T)\rtimes W
$$ 
be the affine Weyl group and the Iwahori-Weyl group. Consider the \emph{length function} 
\begin{eqnarray*}
\ell:\widetilde{W}&\lra& \bbN\\
\nu w&\lmapsto& \sum_{\substack{\alpha\in\Phi^+\\ w^{-1}(\alpha)>0}}|\langle \nu,\alpha\rangle| + \sum_{\substack{\alpha\in\Phi^+ \\ w^{-1}(\alpha)<0}}|\langle \nu,\alpha\rangle+1|.
\end{eqnarray*}
Let $S_a$ be the set of elements of length $1$ which are contained in $W_a$. Then $(W_a,S_a)$ is a Coxeter system. Let $\Omega\subset \widetilde{W}$ be the set of elements of length $0$. This is a subgroup and $\widetilde{W}=W_a\rtimes\Omega$. Finally, denote by $\cB$ the Iwahori group scheme equal to the dilation of $G_{k[\![t]\!]}$ along $B_k$, and for each $s\in S_a$, by $\cP_s$ the parahoric group scheme increasing $\cB$ determined by $s$.

Now let $\lambda\in X_*(T)^+$. Choose a reduced expression of $\lambda w_0\in\widetilde{W}$, i.e. an $(n+1)$-tuple $(s_1,\ldots,s_n,\omega)\in S_a^n\times\Omega$ such that $s_1\cdots s_n\omega=\lambda w_0$ and $n=\ell(\lambda w_0)$. In the next proposition, we denote by 
$\overline{\Fl_G^{\lambda w_0}}$ the \emph{Schubert variety} of $\lambda w_0$ in the \emph{affine flag variety} $\Fl_G:=LG/L^+\cB$, i.e. the closure of 
$\Fl_G^{\lambda w_0}:=L^+\cB\cdot \lambda w_0\subset\Fl_G$. 

\begin{prop}\label{resolD}
The fpqc quotient $\widetilde{X}:=L^+\cP_{s_1}\times^{L^+\cB}\cdots\times^{L^+\cB}L^+\cP_{s_n}/L^+\cB$ is representable by a connected smooth projective scheme over $k$, and it is equipped with a $T$-action by multiplication on the left on the factor $L^+\cP_{s_1}$. The morphism
\begin{eqnarray*}
L^+\cP_{s_1}\times^{L^+\cB}\cdots\times^{L^+\cB}L^+\cP_{s_n}/L^+\cB& \lra & LG/L^+\cB=:\Fl_G\\
{[}p_1,\cdots,p_n{]} & \lmapsto & p_1\cdots p_n\omega
\end{eqnarray*}
factors through $\overline{\Fl_G^{\lambda w_0}}$. The canonical projection 
$$
\Fl_G:=LG/L^+\cB\lra LG/L^+G=:\Gr_G
$$
induces a morphism $\overline{\Fl_G^{\lambda w_0}}\ra\overline{\Gr_G^{\lambda}}=:X$. The composition
$$
\xymatrix{
f:\widetilde{X}\ar[r] & X
}
$$
is surjective, $T$-equivariant, and satisfies $Rf_*\bbF_p=\bbF_p[0]$.
\end{prop}

\begin{proof}
The morphism $f_1:\widetilde{X}\ra \overline{\Fl_G^{\lambda w_0}}$ spelled out in the proposition is nothing but the well-known affine Demazure resolution of the Schubert variety $\overline{\Fl_G^{\lambda w_0}}$ \cite[8.8]{PR08}. 
%moreover it is birational by \cite[9.6 (b)]{PR08}, but we do not need it
It satisfies $R(f_1)_*\bbF_p=\bbF_p[0]$. Indeed, decompose it as
$$
\xymatrix{
\widetilde{X} \ar[r]^<<<<<{f_1'} &\overline{\Fl_G^{\lambda w_0}}^{\nor} \ar[r]^<<<<<{f_1''}  &\overline{\Fl_G^{\lambda w_0}}
}
$$
where $f_1''$ is the normalization. Then $f_1''$ is a universal homeomorphism by \cite[9.7 (a)]{PR08}. Moreover $R(f_1')_*\cO_{\widetilde{X}}=\cO_{\overline{\Fl_G^{\lambda w_0}}^{\nor}}[0]$ by \cite[9.7 (d)]{PR08}, whence $R(f_1')_*\bbF_p=\bbF_p[0]$ by considering the Artin-Schreier short exact sequences 
%for the étale topology !
on $\widetilde{X}$ and on $\overline{\Fl_G^{\lambda w_0}}^{\nor}$. 

On the other hand, the morphism $f_2:\overline{\Fl_G^{\lambda w_0}}\ra\overline{\Gr_G^{\lambda}}=:X$ is the restriction over 
$\overline{\Gr_G^{\lambda}}$ of the canonical projection $\Fl_G\ra\Gr_G$. In particular it is a $G/B$-bundle, whence $R(f_2)_*\bbF_p=\bbF_p[0]$ by proper base change and the Bruhat decomposition of the flag variety $G/B$ (which can be filtered), 
%since the closures of $U$-orbits are again union of $U$-orbits !
cf. \ref{vanishfilaff}.

Thus $Rf_*\bbF_p=R(f_2)_*R(f_1)_*\bbF_p=\bbF_p[0]$.
\end{proof}

\begin{rmrk}\label{resolgen}
The morphism $\widetilde{X}\ra \overline{\Fl_G^{\lambda w_0}}$ in \ref{resolD} is moreover birational, so that it is a \emph{resolution of singularities} of 
the Schubert variety $\overline{\Fl_G^{\lambda w_0}}$, and $\widetilde{X}\ra X$ in \ref{resolD} is the composition of the latter with the $G/B$-fibration 
$\overline{\Fl_G^{\lambda w_0}}\ra\overline{\Gr_G^{\lambda}}$. Instead, we could also have used a $T$-equivariant resolution of singularities of the variety $\overline{\Gr_G^{\lambda}}$ itself, e.g. the affine Demazure resolution of $\overline{\Fl_G^{\lambda}}$ followed by the birational projection $\overline{\Fl_G^{\lambda}}\ra\overline{\Gr_G^{\lambda}}$. 
%for which $Rf_*\bbF_p=\bbF_p[0]$ can certainly be checked by hand, but with a more developped discussion than in the previous $G/B$-fibration case

In fact, this resolution of $\overline{\Gr_G^{\lambda}}$ is a very particular case of the equivariant resolutions of singularities of Schubert varieties in the twisted affine flag variety associated to any connected reductive group over $k(\!(t)\!)$ constructed in \cite{R13}; precisely it is a particular case of \cite[3.2 (i)]{R13}\footnote{for the normalization of the Kottwitz map as in \cite{PR08}, which is opposite to the one in \cite{R13}.}. If the reductive group over $k(\!(t)\!)$ splits over a tamely ramified extension and the order of the fundamental group of its derived subgroup is prime-to-$p$, then any Schubert variety has rational singularities by \cite[8.4]{PR08}; since `having rational singularities' is an intrinsic notion by \cite[Th. 1]{CR11} (see also \cite{K20}), 
%[Th. 9.13]
then in this case all the resolutions $f$ from \cite{R13} satisfy $Rf_*\bbF_p=\bbF_p[0]$ (using Artin-Schreier).
\end{rmrk}

\subsection{$\bbF_p$-direct images with compact support of the MV-cycles}

\begin{thm}\label{thmcompute}
Let $(\nu,\lambda)\in X_*(T)\times X_*(T)^+$. Then
$$
R\Gamma_c(S_{\nu}\cap \overline{\Gr_G^{\lambda}},\mathbb{F}_p)= 
\left\{ \begin{array}{ll} 
\mathbb{F}_p|_{\{w_0(\lambda)\}}[0] & \quad \textrm{if $\nu=w_0(\lambda)$}, \\
0 & \quad \textrm{otherwise}.
\end{array}
\right.
$$

More generally, let $(c,\lambda)\in \pi_0(\Gr_L)\times X_*(T)^+$. Let
$$
q_{c,\lambda}:S_c\cap \overline{\Gr_G^{\lambda}}\lra \Gr_L^c\cap \overline{\Gr_G^{\lambda}}
$$
be the morphism of $k$-schemes defined by the diagram
$$
\xymatrix{
&\Gr_P \ar[dl] \ar[dr] & \\
\Gr_L & & \Gr_G.
}
$$
Then
$$
R(q_{c,\lambda})_!\mathbb{F}_p= 
\left\{ \begin{array}{ll} 
\mathbb{F}_p|_{\overline{\Gr_{L,w_0^Lw_0(\lambda)}}}[0] & \quad \textrm{if $c=c(w_0(\lambda))$}, \\
0 & \quad \textrm{otherwise}.
\end{array}
\right.
$$
\end{thm}

%Could be generalized using \label{resolgen}

\begin{proof}
Let $\eta_T:\bbG_m\ra T$ be a regular dominant cocharacter. We start by applying \ref{vanishstrategy} to $X:=\overline{\Gr_G^{\lambda}}$ equipped the $\bbG_m$-action $\eta_T(\bbG_m)$ obtained by restriction of the adjoint $T$-action along $\eta_T$; it does apply thanks to \ref{decompfil} (2) combined with \cite[\S 1.5]{Z17}, and \ref{resolD}. 

Recall from \cite[Th. 3.2 (a)]{GeometricSatake} that the MV-cycle $S_{w_0(\lambda)}\cap \overline{\Gr_G^{\lambda}}$ is $0$-dimensional. Hence
$$
(X_{w_0(\lambda)})_{\red}=S_{w_0(\lambda)}\cap \overline{\Gr_G^{\lambda}}=\{w_0(\lambda)\}\subset X_{\red}
$$
is closed, and the theorem in the case of the torus $T$ follows.

Next let $L$ be a standard Levi. We have the canonical commutative diagram
$$
\xymatrix{
\Gr_B \ar[d] \ar[r] &\Gr_P \ar[d] \ar[r] & \Gr_G \\
\Gr_T \ar[r] & \Gr_L.
}
$$
%whose square is even cartesian after passing to the reduced ind-schemes, since the vertical geometric fibres are connected
It shows that for each $c\in\pi_0(\Gr_L)$,
$$
S_c(k)=\bigcup_{\nu\in X_*(T)\cap \Gr_L^c}S_{\nu}(k)\quad \subset \Gr_G(k).
$$
Intersecting with $X=\overline{\Gr_G^{\lambda}}\subset\Gr_G$ we get
$$
X_c(k)=\bigcup_{\nu\in X_*(T)\cap \Gr_L^c\cap X}X_{\nu}(k)\quad \subset X(k).
$$
%(recall that $X\subset \Gr_G$ is closed and $\eta_T(\bbG_m)$-stable, the intersection $S_{\nu}(k)\cap X(k)=X_{\nu}(k)$ is nonempty if and only if $\nu\in X(k)$) 
Consequently, the subscheme $(X_c)_{\red}\subset X$ is $\eta_T(\bbG_m)$-stable, and the reduced connected components of its attractor are realized by the subschemes $(X_{\nu})_{\red}$, $\nu\in X_*(T)\cap (\Gr_L^c\cap X)$. In particular, by  \ref{decompfil} (2), there exists at least one nonempty 
closed $(X_{\nu})_{\red}\subset (X_c)_{\red}$.

Now let $\eta_L:\bbG_m\ra T$ be a dominant cocharacter whose centralizer in $G$ is $L$, and equip $X:=\overline{\Gr_G^{\lambda}}$ with the $\bbG_m$-action 
$\eta_L(\bbG_m)$ obtained by restriction of the adjoint $T$-action along $\eta_L$. Thanks to \ref{decompfil} (2) combined with \cite[\S 1.5]{Z17},  there exists at least one nonempty $(X_{c_0})_{\red}:=(X_c)_{\red}\subset X$ which is closed. Choosing $(X_{\nu_0})_{\red}\subset (X_{c_0})_{\red}$ nonempty and closed, then we get $(X_{\nu_0})_{\red}\subset X_{\red}$ nonempty and closed, so that $\nu_0=w_0(\lambda)$ by the torus case. Hence $c_0=c(w_0(\lambda))$. And by \ref{Lintersection} (2) below, 
$$
|X_{c(w_0(\lambda))}^0|=|\Gr_L^{c(w_0(\lambda))} \cap X|=|\overline{\Gr_{L,w_0^L w_0 (\lambda)}}|.
$$
The theorem in the case of the standard Levi $L$ follows by \ref{vanishstrategy}, which applies thanks to \ref{resolD}.
\end{proof}

\begin{lem} \label{Lintersection}
Let $c \in \pi_0(\Gr_L)$ and $\lambda \in X_*(T)^+$. 
\leavevmode
\begin{enumerate}[{\normalfont (i)}]
\item If $\lambda \in c$ then $\Gr_L^c \cap \Gr_G^{\leq \lambda} = \Gr_L^{\leq \lambda}$.
\item If $w_0 (\lambda) \in c$ then $\Gr_L^c \cap \Gr_G^{\leq \lambda} = \Gr_L^{\leq w_0^L w_0 (\lambda)}$.
\end{enumerate}
\end{lem}

\begin{proof} Let $\Delta^\vee \subset \Phi^\vee$ be the set of simple coroots of $G$ with respect to the pair $(B,T)$, and let $\Delta_L^\vee \subset \Delta^\vee$ be the subset of simple coroots of the Levi $L$ with respect to $(B\cap L,T)$. By the Cartan decomposition
$$\Gr_L \cap \Gr_G^{\leq \lambda} = \bigcup_{\substack{\lambda' \in X_*(T)^+ \\ \lambda' \leq \lambda}} \: \: \bigcup_{\mu \in X_*(T)_{+/L} \cap W \lambda'} \Gr_L^{\mu}.$$

As $\Gr_L^c \cap \Gr_G^{\leq \lambda}\subset \Gr_L$ is closed and $L^+L$-stable, to prove (i) it suffices to show that, for $\lambda \in c$ and $\mu$ as above,  $\Gr_L^c \cap \Gr_L^{\mu} = \emptyset$ unless $\mu \leq_L \lambda$. To prove this, suppose $\Gr_L^c \cap \Gr_L^{\mu} \neq \emptyset$. Then $\lambda - \mu \in \mathbb{Z}\Delta_L^\vee$, and moreover since $\mu \in W \lambda'$ we have $\lambda - \mu \in \mathbb{N} \Delta^\vee$. Because $\Delta^\vee$ is linearly independent then $\lambda - \mu \in \mathbb{Z} \Delta_L^\vee \cap \mathbb{N} \Delta^\vee = \mathbb{N} \Delta_L^\vee$. Thus $\mu \leq_L \lambda$ and hence the claim follows. Finally, (ii) can be proved  similarly, since then  $\Gr_L^c \cap \Gr_L^{\mu} \neq \emptyset$ implies $w_0(\lambda) - w_0^L(\mu) \in \mathbb{Z}\Delta_L^\vee$ and $\mu \in W \lambda'$ implies $w_0^L(\mu)-w_0(\lambda) \in \mathbb{N} \Delta^\vee$, and hence $w_0^L(\mu)-w_0(\lambda) \in \mathbb{Z}\Delta_L^\vee \cap \mathbb{N} \Delta^\vee = \mathbb{N} \Delta_L^\vee$. 
%the fact that $w_0 (\lambda) \leq_G \mu$ for all $\mu \in W \lambda'$.
\end{proof}

Finally, we record from the proof of \ref{thmcompute} (and \ref{trivred}):

\begin{cor} \label{closedSc}
For all $\lambda\in X_*(T)^+$,
$$
S_{c(w_0(\lambda))}\cap\Gr_G^{\leq \lambda}
%= (\Gr_L^{c(w_0(\lambda)} \cap \Gr_G^{\leq \lambda})_{\red} 
= \Gr_{L}^{\leq w_0^L w_0 (\lambda)}.
$$
\end{cor}

\section{Hyperbolic localization on the affine Grassmannian} \label{absolutesection}

\subsection{Perverse $\bbF_p$-sheaves on the affine Grassmannian}
For a separated scheme $X$ of finite type over $k$ let $P_c^b(X, \bbF_p)$ be the abelian category of perverse $\bbF_p$-sheaves on $X$ as defined in \cite[\S 2]{modpGr}. This is an abelian subcategory of $D_c^b(X, \bbF_p)$ in which all objects have finite length. The definition of perverse sheaves extends to ind-schemes of ind-finite type as in \cite[3.13]{modpGr}. 

Let $P_{L^+G}(\Gr_G, \bbF_p) \subset D_c^b(\Gr_G, \bbF_p)$ be the full abelian subcategory of $L^+G$-equivariant perverse $\bbF_p$-sheaves on $\Gr_G$. By \cite[1.1]{modpGr}, the category $P_{L^+G}(\Gr_G, \bbF_p)$ is symmetric monoidal and the functor 
$$
\xymatrix{
H=\bigoplus_{i\in\bbZ}R^i\Gamma:(P_{L^+G}(\Gr_G,\bbF_p),*)\ar[r] & (\Vect_{\bbF_p},\otimes)
}
$$
is an exact faithful tensor functor. The definition of the convolution product $*$ will be reviewed in Subsection \ref{recallconv}.

%The functor of tensor endomorphisms of $H$ is represented by an affine monoid scheme $M_G$ over $\bbF_p$, and there is a natural equivalence of tensor categories
%$$(P_{L^+G}(\Gr_G, \bbF_p), *) \cong (\Rep_{\bbF_p}(M_G), \otimes).$$ 
By \cite[1.5]{modpGr}, the simple objects in $P_{L^+G}(\Gr_G, \bbF_p)$ are the shifted constant sheaves:
$$\IC_\lambda = \mathbb{F}_p[2 \rho(\lambda)] \in P_c^b(\Gr_G^{\leq \lambda}, \bbF_p), \quad \lambda \in X_*(T)^+.$$ Furthermore, if $\lambda_i \in X_*(T)^+$ then by \cite[1.2]{modpGr} there is a natural isomorphism
$$\IC_{\lambda_1} * \IC_{\lambda_2} \cong \IC_{\lambda_1+\lambda_2}.$$

\subsection{The hyperbolic localization functor}

\begin{defn}
Let $\nu \in X_*(T)$ and $\cF^{\bullet}\in D_c^b(\Gr_G,\bbF_p)$. Denote by $s_\nu \colon S_\nu \rightarrow \Gr_G$ the ind-immersion of the corresponding connected component of $(\Gr_B)_{\red}$ and define
$$
R\Gamma_c(S_{\nu}, \cF^{\bullet}) : = R\Gamma_c(S_\nu ,Rs_\nu^*\cF^{\bullet})\in D_c^b(\Vect_{\bbF_p}),
$$
$$
\textrm{and}\quad \forall i\in\bbZ,\quad H_c^i(S_{\nu}, \cF^{\bullet}) : = H^i\big(R\Gamma_c(S_{\nu}, \cF^{\bullet}) \big)=H_c^i(S_\nu, Rs_\nu^*\cF^{\bullet})\in\Vect_{\bbF_p}.$$ 
\end{defn}

\begin{thm} \label{computetop}
Let $\nu\in X_*(T)$ and $\lambda\in X_*(T)^+$.
\begin{itemize}
\item[(i)] We have
$$
H^{2\rho(\nu)}_c(S_{\nu},\IC_{\lambda})  = 
\left\{ \begin{array}{ll} 
H^0(\{w_0(\lambda)\}, \mathbb{F}_p) = \bbF_p & \quad \textrm{if $\nu=w_0(\lambda)$}, \\
0 & \quad \textrm{otherwise}.
\end{array}
\right.
$$ \item[(ii)] If $i \neq 2 \rho(\nu)$ then 
$$H^{i}_c(S_{\nu},\IC_{\lambda})  = 0.$$
%\item[(iii)] We have
%$$R\Gamma_c(S_{\nu}, \IC_{\lambda})=H^{2\rho(\nu)}(S_{\nu},\IC_{\lambda})[-2 \rho(\nu)].$$
\item[(iii)] If  $\cF^{\bullet}\in P_{L^+G}(\Gr_G,\bbF_p)$, then
$$
R\Gamma_c(S_{\nu}, \cF^{\bullet}) \in D_c^{\leq 2\rho(\nu)}(\Vect_{\bbF_p}) \cap D_c^{\geq 2\rho(\nu)}(\Vect_{\bbF_p})=\Vect_{\bbF_p}[-2\rho(\nu)].
$$
 \end{itemize}
\end{thm}

\begin{proof}
Since $\IC_{\lambda}$ is the shifted constant sheaf $\mathbb{F}_p[2 \rho(\lambda)]$ supported on $\Gr^{\leq \lambda}_G$ then parts (i) and (ii) follow immediately from \ref{thmcompute}. To prove part (iii), by d\'{e}vissage we can assume that $\cF^{\bullet}=\IC_{\lambda}$ for some $\lambda\in X_*(T)^+$. Then part (iii) follows from (i) and (ii).
\end{proof}

\begin{rmrk} We claim that $$H^{2\rho(\nu)}_c(S_{\nu},\IC_{\lambda})\cong H_c^{2\rho(\nu+\lambda)}(S_{\nu}\cap\Gr_G^{\lambda},\bbF_p),$$ which is also true for characteristic $0$ coefficients, see e.g. \cite[proof of 1.5.13]{BR18}. To prove the claim, note that it suffices to show that the canonical map
$$
\xymatrix{
H_c^{2\rho(\nu+\lambda)}(S_{\nu}\cap\Gr_G^{\lambda},\bbF_p)\ar[r] & H_c^{2\rho(\nu+\lambda)}(S_{\nu}\cap\Gr_G^{\leq\lambda},\bbF_p)
}
$$
is an isomorphism. If $\nu = w_0 (\lambda)$ then $S_{\nu}\cap\Gr_G^{\lambda} = S_{\nu}\cap\Gr_G^{\leq\lambda} = \{\nu\}$, so the claim follows in this case. If $\nu \neq w_0 (\lambda)$ then by \ref{thmcompute} we must show that $H_c^{2\rho(\nu+\lambda)}(S_{\nu}\cap\Gr_G^{\lambda},\bbF_p) = 0$. Note that $\dim S_{\nu}\cap\Gr_G^{\lambda} = \rho(\nu + \lambda) > 0$ by \cite[1.5.2 3.]{BR18}. The desired vanishing then follows from the following general fact (cf. \cite[7.2.11]{FuEtale}): if $X$ is a separated scheme of finite type over $k$, then
$$
\forall i>\dim X,\quad H_c^{i}(X,\bbF_p)=0.
$$
\end{rmrk}

\subsection{An alternative description of the hyperbolic localization functor}
\begin{defn}
Let $\nu \in X_*(T)$ and $\cF^{\bullet}\in D_c^b(\Gr_G,\bbF_p)$. Denote by $i_{\nu} \colon \{\nu\} \rightarrow \Gr_G$ the inclusion of the $k$-point $\nu(t)$ and define
$$R\Gamma(\{\nu\}, \mathcal{F}^\bullet)  := Ri_{\nu}^* \mathcal{F}^\bullet \in D_c^b(\Vect_{\bbF_p}),$$
$$
\textrm{and}\quad \forall i\in\bbZ,\quad H^i(\{\nu\}, \cF^{\bullet}) : = H^i\big(R\Gamma(\{\nu\}, \cF^{\bullet}) \big) = H^i(Ri_{\nu}^* \mathcal{F}^\bullet) \in\Vect_{\bbF_p}.$$
\end{defn}

\begin{lem} \label{halfperverse}
Let $\nu \in X_*(T)$ and $\cF^{\bullet}\in P_{L^+G}(\Gr_G,\bbF_p)$. 
\begin{itemize}
\item[(i)] If $\lambda \in X_*(T)^+$ then $$H^{2 \rho(\nu)}(\{\nu\}, \IC_\lambda)= 
\left\{ \begin{array}{ll} 
H^{0}(\{\nu\}, \mathbb{F}_p) = \bbF_p & \quad \textrm{if $\nu=w_0(\lambda)$}, \\
0 & \quad \textrm{otherwise}.
\end{array}
\right.$$
%\item[(ii)] We have $$R\Gamma(\{\nu\}, \cF^{\bullet}) \in D_c^{\leq 2\rho(\nu)}(\Vect_{\bbF_p}).$$
\item[(ii)] If $H^i(\{\nu\}, \mathcal{F}^\bullet) \neq 0$ then $$i \equiv 2\rho(\nu) \mod 2.$$
\end{itemize}
 \end{lem}

\begin{proof}
For part (i), we have
$$H^{2 \rho(\nu)}(\{\nu\}, \IC_\lambda)= H^{2\rho(\nu + \lambda)}(\{\nu\} \cap \Gr_G^{\leq \lambda}, \mathbb{F}_p).$$ 
This is zero unless $\{\nu\} \in \Gr_G^{\leq \lambda}$ and $2\rho(\nu+\lambda) = 0$, in which case $w_0(\lambda)\leq \nu\leq \lambda$  and $2\rho(\nu-w_0(\lambda)) = 0$, i.e. $\nu = w_0 (\lambda)$.
 
By d\'{e}vissage, to prove part (ii) we can assume that $\cF^{\bullet}=\IC_{\lambda}$ for some $\lambda\in X_*(T)^+$. Then for all $i\in\bbZ$ we have
$$
H^i(\{\nu\}, \IC_\lambda) = H^{i+2\rho(\lambda)}(\{\nu\} \cap \Gr_G^{\leq \lambda}, \mathbb{F}_p).
$$
If this is nonzero then $\{\nu\} \in \Gr_G^{\leq \lambda}$ and $i+2\rho(\lambda) = 0$, so $\rho(\lambda-\nu)$ is an integer and
$$
i + 2\rho(\nu) = i + 2\rho(\lambda) - 2\rho(\lambda-\nu)  \equiv 0 \mod 2.
$$
\end{proof}

\begin{thm} \label{wtequivalence}
For $\nu \in X_*(T)$ there is an isomorphism of functors 
$$
H^{2\rho(\nu)}_c(S_\nu, \cdot) \xrightarrow{\sim} H^{2 \rho(\nu)}(\{\nu\}, \cdot ) \colon P_{L^+G}(\Gr_G, \mathbb{F}_p) \longrightarrow \Vect_{\mathbb{F}_p}.
$$
\end{thm}
%Remark: not true if $\nu \in X_*(T)\setminus X_*(T)_-$, because then H_c^{2 \rho(\nu)}(S_\nu, \cdot ) =0 while H^{2 \rho(\nu)}(L_\nu, \cdot ) does not vanish if $\rho(\nu)=0$.

\begin{proof}
By the adjunction between $Ri_\nu^*$ and $Ri_{\nu*}$ there is a natural map $H_c^{2 \rho(\nu)}(S_\nu, \mathcal{F}^\bullet ) \rightarrow H^{2 \rho(\nu)}(\{\nu\}, \mathcal{F}^\bullet)$. If $\mathcal{F} = \IC_{\lambda}$ for $\lambda \in X_*(T)^+$ then it is an isomorphism by \ref{computetop} (i) and \ref{halfperverse} (i). For the general case, note that $H^{2 \rho(\nu) - 1}(\{\nu\}, \mathcal{F}^\bullet ) = H^{2 \rho(\nu)+1}(\{\nu\}, \mathcal{F}^\bullet ) = 0$ for all $\mathcal{F}^\bullet \in P_{L^+G}(\Gr_G, \mathbb{F}_p)$ by \ref{halfperverse} (ii). Since $H_c^{2 \rho(\nu)+1}(S_\nu, \mathcal{F}^\bullet) = 0$ for all $\mathcal{F}^\bullet \in P_{L^+G}(\Gr_G, \mathbb{F}_p)$ by \ref{computetop} (iii), then by induction on the length of $\mathcal{F}^\bullet$ and the five lemma we see that the map $H_c^{2 \rho(\nu)}(S_\nu, \mathcal{F}^\bullet ) \rightarrow H^{2 \rho(\nu)}(\{\nu\}, \mathcal{F}^\bullet)$ is an isomorphism in general. 
\end{proof}

\section{The total weight functor} \label{totalsection}

\subsection{The definition of the total weight functor}

\begin{defn} \label{wtdef1}
For $\nu \in X_*(T)$, the \emph{weight functor associated to $\nu$} is 
$$
\xymatrix{
F_ \nu : = H^{2\rho(\nu)}_c(S_\nu, \cdot) \xrightarrow{\sim} H^{2 \rho(\nu)}(\{\nu\}, \cdot ) \colon P_{L^+G}(\Gr_G, \mathbb{F}_p) \ar[r] & \Vect_{\mathbb{F}_p}.
}
$$
\end{defn}

\begin{prop}\label{nunotin-} 
The functor $F_\nu$ is exact. Furthermore, if $\nu \notin X_*(T)_{-}$ then $F_\nu = 0$. 
\end{prop} 

\begin{proof} Exactness follows from \ref{computetop} (iii). Since for $\nu \notin X_*(T)_{-}$ we have $F_\nu(\mathcal{F}^\bullet) = 0$ for all simple 
$\mathcal{F}^\bullet\in P_{L^+G}(\Gr_G, \mathbb{F}_p)$ by \ref{computetop} (i), we may conclude by induction on the length that $F_\nu = 0$ in this case. 
\end{proof}

\begin{nota}\label{notagradsp}
Given an abstract abelian monoid $A$, we will denote by $(\Vect_{\bbF_p}(A),\otimes)$ the symmetric monoidal category of finite dimensional $A$-graded $\bbF_p$-vector spaces equipped with the tensor product 
$$
\bbF_p(a)\otimes \bbF_p(b):=\bbF_p(a+b),
$$
where $\bbF_p(a)$ denotes the vector space $\bbF_p$ placed in `degree' $a\in A$. 
%(i.e. $\bbF_p(a):=\bbF_p[-a]$ -perhaps doesn't make sense if $a$ is not invertible).
\end{nota}

\begin{defn}
The \emph{total weight functor} is
$$
\xymatrix{
F_-:= \bigoplus_{\nu\in X_*(T)_-} F_{\nu}\colon P_{L^+G}(\Gr_G, \mathbb{F}_p)\ar[r] &\Vect_{\mathbb{F}_p}(X_*(T)_-).
}
$$
\end{defn}

\begin{rmrk}
Recall from Subsection \ref{Iwasec} the diagram 
$$
\xymatrix{
&\Gr_B \ar[dl]_{q} \ar[dr]^p & \\
\Gr_T & & \Gr_G.
}
$$ Since
$$(\Gr_T)_{\text{red}} = \coprod_{\nu \in X_*(T)} \{\nu\}$$ then $F_-$ can be obtained from the functor
$$\xymatrix{Rq_! \circ Rp^* \colon P_{L^+G}(\Gr_G, \mathbb{F}_p) \ar[r] & D_c^b(\Gr_T, \mathbb{F}_p)}$$ by taking the direct sum of the stalks over the $\{\nu\}$ in degree $2 \rho(\nu)$. This identifies $F_-$ with the $T$-constant term functor $\CT_T^G$ defined in \ref{CTdef}.
\end{rmrk}

\subsection{Relation to the Satake equivalence} \label{section_relationsat}

Recall the exact faithful symmetric monoidal functor
$$
\xymatrix{
H=\bigoplus_{i\in\bbZ}R^i\Gamma:(P_{L^+G}(\Gr_G,\bbF_p),*)\ar[r] & (\Vect_{\bbF_p},\otimes)
}
$$
from \cite[6.11, 7.11]{modpGr}. Our goal in this subsection is to construct a natural isomorphism between $H$ and $F_-$ composed with the forgetful functor $\Vect_{\mathbb{F}_p}(X_*(T)_-) \ra \Vect_{\mathbb{F}_p}$.

\begin{rmrk}
In the case of characteristic $0$ coefficients, Baumann and Riche construct an isomorphism between $H$ and $\bigoplus_{\nu\in X_*(T)} F_\nu$ in the proof of \cite[1.5.9]{BR18}. 
%Here this is good to quote this reference rather than the original \cite[3.6]{MV07}, since as noticed by Scholze in his online course there is a wrong argument regarding the filtrations (they are equal instead of complementary).
In our proof of \ref{Zwt} below we use \ref{wtequivalence}, which is unique to $\mathbb{F}_p$-sheaves, to compare the functors $H$ and $F_-$.
\end{rmrk}

By \cite[6.9]{modpGr}, $R^i\Gamma (\mathcal{F}^\bullet)=0$ for all $\mathcal{F}^\bullet \in P_{L^+G}(\Gr_G, \mathbb{F}_p)$ and $i>0$. Set $\bbZ_-:=\bbZ_{\leq 0}$. For all $i\in \bbZ_-$, the adjunction between $Ri_{\nu}^*$ and $Ri_{\nu*}$ induces a natural transformation of functors
$$
\xymatrix{
R^i\Gamma\ra \bigoplus_{\substack{\nu\in X_*(T)_-\\2\rho(\nu)=i}}H^{2\rho(\nu)}(\{\nu\}, \cdot)
}
$$ 
from  $P_{L^+G}(\Gr_G, \mathbb{F}_p)$ to $\Vect_{\mathbb{F}_p}$. Hence there is a natural transformation of functors 
$$
\xymatrix{
H=\bigoplus_{i\in\bbZ_-}R^i\Gamma \ra \bigoplus_{i\in\bbZ_-}\bigoplus_{\substack{\nu\in X_*(T)_-\\2\rho(\nu)=i}}F_{\nu}=\bigoplus_{\nu\in X_*(T)_-} F_{\nu}
}
$$
from $P_{L^+G}(\Gr_G, \mathbb{F}_p)$ to $\Vect_{\mathbb{F}_p}$.

%For the rest of this section when we write $\oplus_\nu$ we will always mean the sum over $\nu \in X_*(T)_-$.

\begin{thm} \label{HFtheorem} \label{Zwt}
The natural transformation of functors 
$$
\xymatrix{
H \ra \bigoplus_{\nu\in X_*(T)_-} F_{\nu}\colon P_{L^+G}(\Gr_G, \mathbb{F}_p)\ar[r] &\Vect_{\mathbb{F}_p}
}
$$
is an isomorphism. In particular, for all $i\in\bbZ$
%in fact, equivalently, since the map is $\bbZ$-graded by construction
it restricts to an isomorphism
$$
\xymatrix{
R^i\Gamma \cong \bigoplus_{\substack{\nu\in X_*(T)_-\\2\rho(\nu)=i}} F_{\nu} \colon P_{L^+G}(\Gr_G, \mathbb{F}_p) \ar[r] &\Vect_{\mathbb{F}_p}.
}
$$

\end{thm}

\begin{proof}
Let $\lambda\in X_*(T)^+$. Combining \cite[6.9]{modpGr} and \ref{computetop} (i), taking the stalk at $\{w_0(\lambda)\}$ defines an isomorphism
in $\Vect_{\mathbb{F}_p}$
$$
H(\IC_{\lambda})=R^{-2\rho(\lambda)}\Gamma(\IC_{\lambda})=H^{-2\rho(\lambda)}(\Gr_{\leq\lambda},\bbF_p[2 \rho(\lambda)])
$$
$$
\xrightarrow{\sim}H^{2\rho(w_0 (\lambda))}(Ri_{w_0( \lambda)}^* \mathbb{F}_p[2 \rho(\lambda)]) =F_{w_0(\lambda)}(\IC_{\lambda}).
$$
Thus since $F_\nu(\IC_\lambda) = 0$ if $\nu \neq w_0(\lambda)$ then the natural map 
$$
\xymatrix{
H(\mathcal{F}^\bullet) \rightarrow \bigoplus_{\nu\in X_*(T)_-} F_\nu(\mathcal{F}^\bullet)
}
$$ 
is an isomorphism if $\mathcal{F}^\bullet$ is simple. Now $H$ is exact by \cite[6.11]{modpGr} and each $F_\nu$ is exact by \ref{nunotin-}. Hence it follows by induction on the length of $\mathcal{F}^\bullet$ that the above map is an isomorphism in general.
\end{proof}

By \ref{HFtheorem}, composing $F_-$ with the forgetful functor $\Vect_{\mathbb{F}_p}(X_*(T)_{-}) \rightarrow \Vect_{\mathbb{F}_p}$ gives $H$.

\begin{rmrk}
Using the method in \cite[3.6]{GeometricSatake} one can show that the decomposition $H \cong \oplus_{\nu\in X_*(T)_-} F_{\nu}$ is independent of the choice of the pair $(T,B)$.
\end{rmrk}

\subsection{Recollections on convolution} \label{recallconv}

We first recall the definition of the convolution product in $P_{L^+G}(\Gr_G, \mathbb{F}_p)$ following \cite[\S 6.2]{modpGr}. There is a diagram
\begin{equation} \label{localconvdiagram} \xymatrix{
\Gr_G \times \Gr_G & LG \times \Gr_G \ar[l]_(.45)p \ar[r]^(.45)q & LG \overset{L^+G}{\times} \Gr_G \ar[r]^(.6)m & \Gr_G
}.\end{equation} Here $p$ is the quotient map on the first factor, $q$ is the quotient by the diagonal action of $L^+G$, and $m$ is induced by multiplication in $LG$. We set
$$\Gr_G \overset{\sim}{\times} \Gr_G := LG \overset{L^+G}{\times} \Gr_G.$$ For $\mathcal{F}_1^\bullet$, $\mathcal{F}_2^\bullet \in P_{L^+G}(\Gr_G, \mathbb{F}_p)$ there exists a unique perverse sheaf $$\mathcal{F}_1^\bullet \overset{\sim}{\boxtimes} \mathcal{F}_2^\bullet \in P_c^b(\Gr_G \overset{\sim}{\times} \Gr_G, \mathbb{F}_p)$$ such that
$$Rp^*( \mathcal{F}_1^\bullet \overset{L}{\boxtimes} \mathcal{F}_2^\bullet) \cong Rq^* (\mathcal{F}_1^\bullet \overset{\sim}{\boxtimes} \mathcal{F}_2^\bullet).$$ The convolution product is
$$\mathcal{F}_1^\bullet * \mathcal{F}_2^\bullet := Rm_!(\mathcal{F}_1^\bullet \overset{\sim}{\boxtimes} \mathcal{F}_2^\bullet).$$ Note that because $\Gr_G \overset{\sim}{\times} \Gr_G$ is ind-projective we have $m_! = m_*$.

We now recall the construction of the monoidal structure on $H$ following \cite[\S 7]{modpGr}. Let $X = \mathbb{A}^1$. The construction uses the Beilinson-Drinfeld Grassmannians $\Gr_{G,X^I}$ and the global convolution Grassmannians $\tilde{\Gr}_{G,X^I}$ for $I=\{*\}$ and $I=\{1,2\}$ (see also \cite[\S 3.1]{Z17}). There is a convolution morphism $m_I \colon \tilde{\Gr}_{G,X^I} \rightarrow \Gr_{G,X^I}$ and a projection $f_I \colon \Gr_{G,X^I} \rightarrow X^I$. Since $X=\bbA^1$, for $I=\{*\}$ there are canonical isomorphisms
$$
\Gr_{G,X}\cong\Gr_G\times X,\quad \tilde{\Gr}_{G,X}\cong (\Gr_G \overset{\sim}{\times} \Gr_G)\times X,\quad m_{\{*\}}=m\times\id,\quad f_{\{*\}}=\pr_2.
$$
So in the sequel we keep the notation $I$ for the set $\{1,2\}$ only. Let $U \subset X^2$ be the complement of the image of the diagonal embedding $\Delta \colon X \rightarrow X^2$. Then we have the following commutative diagram with Cartesian squares:

\begin{equation} \label{BDdiagram} \xymatrix{
\Gr_G \times \Gr_G \times U \ar[r] \ar[d]^{\id} & \tilde{\Gr}_{G,X^2} \ar[d]^{m_I} & (\Gr_G \overset{\sim}{\times} \Gr_G)\times X \ar[d]^{m \times \id}  \ar[l] \\
\Gr_G \times \Gr_G \times U \ar[r]^(.65){j_I} \ar[d] & \Gr_{G,X^2} \ar[d]^{f_I} & \Gr_G \times X \ar[d] \ar[l]_(.50){i_I}  \\
U \ar[r] & X^2 & \ar[l]_{\Delta} X.
} \end{equation}

\noindent Let $\tau \colon \Gr_{G,X}=\Gr_G\times X \rightarrow \Gr_G$ be the projection and let
$$\tau^{\circ} := R\tau^{*}[1] \colon D_c^b(\Gr_G,\mathbb{F}_p) \rightarrow D_c^b(\Gr_{G,X}, \mathbb{F}_p).$$

\noindent Fix $\mathcal{F}_1^\bullet$, $\mathcal{F}_2^\bullet \in P_{L^+G}(\Gr_G, \mathbb{F}_p)$. By \cite[7.6, 7.10]{modpGr} there is a perverse sheaf $$\mathcal{F}_{1,2}^\bullet := \tau^{\circ} \mathcal{F}_1^\bullet \overset{\sim}{\boxtimes} \tau^{\circ} \mathcal{F}_2^\bullet \in P_c^b(\tilde{\Gr}_{G,X^2}, \mathbb{F}_p)$$ such that for $x_1$, $x_2 \in X(k)$,
\begin{equation} \label{eq1}
\restr{H^{n-2}(Rf_{I,!}(Rm_{I,!}\mathcal{F}^\bullet_{1,2}))}{(x_1,x_2)} \cong \left\{
        \begin{array}{ll}
            \bigoplus_{i+j=n}R^{i}\Gamma(\mathcal{F}^\bullet_1) \otimes R^{j}\Gamma(\mathcal{F}^\bullet_2) & \quad \textrm{if $x_1 \neq x_2$} \\
            R^n\Gamma(\mathcal{F}_1^\bullet * \mathcal{F}_2^\bullet) & \quad \textrm{if $x_1 = x_2$}.
        \end{array}
    \right.
\end{equation}
The sheaf $H^{n-2}(Rf_{I,!}(Rm_{I,!}\mathcal{F}^\bullet_{1,2}))$ is constant by \cite[7.9]{modpGr}. Therefore, by summing (\ref{eq1}) over $n$ we get an isomorphism
$$H(\mathcal{F}_1^\bullet * \mathcal{F}_2^\bullet) \cong H(\mathcal{F}_1^\bullet) \otimes H(\mathcal{F}_2^\bullet).$$ This gives $H$ the structure of a monoidal functor.

We finally recall that the associativity constraint in $(P_{L^+G}(\Gr_G, \mathbb{F}_p),*)$ is constructed using the one of the bifunctor $\overset{L}{\boxtimes}$ and proper base change \cite[6.8]{modpGr}, and the commutativity constraint as follows. There is a morphism $\Gr_{G,X^2} \rightarrow \Gr_{G,X^2}$ which swaps the factors in $X^2$. Using that this morphism restricts to the identity map over $\Delta(X)$, it is shown in the proof of \cite[7.11]{modpGr} that there is a canonical isomorphism
$$\restr{j_{I, !*}(\tau^{\circ} \mathcal{F}_1^\bullet \overset{L}{\boxtimes} \tau^{\circ} \restr{\mathcal{F}_2^\bullet}{U})}{\Delta(X)} \cong \restr{j_{I, !*}(\tau^{\circ} \mathcal{F}_2^\bullet \overset{L}{\boxtimes} \tau^{\circ} \restr{\mathcal{F}_1^\bullet}{U})}{\Delta(X)}.$$
On the other hand, we have the following:
\begin{prop} \label{conv1}
There is a canonical isomorphism
$$
\tau^{\circ} (\cF_1^\bullet * \cF_2^\bullet) \cong Ri_I^* \circ j_{I,!*} (\tau^{\circ} \cF_1^\bullet \overset{L}{\boxtimes} \restr{\tau^{\circ} \cF_2^\bullet}{U})[-1].
$$
\end{prop}
\begin{proof}
By the arguments in the proof of \cite[7.10 (ii)]{modpGr}, there is a canonical isomorphism
$$Ri_I^* (Rm_{I,!} (\mathcal{F}_{1,2}^\bullet))[-1] \cong \tau^{\circ}(\cF_1^\bullet * \cF_2^\bullet).$$ On the other hand, by \cite[7.8]{modpGr} we have
\begin{equation} \label{C78}
Rm_{I,!} (\mathcal{F}_{1,2}^\bullet) \cong j_{I,!*} (\tau^{\circ} \cF_1^\bullet \overset{L}{\boxtimes} \restr{\tau^{\circ} \cF_2^\bullet}{U}).
\end{equation}
\end{proof}
%By \cite[7.8]{modpGr} there is a canonical isomorphism
%\begin{equation} \label{eq10} Rm_{I,!} (\mathcal{F}_{1,2}^\bullet) \cong j_{I, !*}(\tau^{\circ} \mathcal{F}_1^\bullet \overset{L}{\boxtimes} \tau^{\circ} 
%\restr{\mathcal{F}_2^\bullet}{U}).
%\end{equation} 
\noindent Consequently, we get a commutativity isomorphism
$$\mathcal{F}_1^\bullet * \mathcal{F}_2^\bullet \cong \mathcal{F}_2^\bullet * \mathcal{F}_1^\bullet.$$ In order to make this commutativity isomorphism compatible with that of $\otimes$ it must be modified by certain sign changes which depend on the parities of the dimensions of the strata occurring in the support of the $\mathcal{F}_i^\bullet$; see the proof of \cite[7.11]{modpGr} for more details.

\subsection{Compatibility with convolution}

\begin{rmrk}
In this subsection we use \ref{wtequivalence} in order to take $H^{2\rho(\nu)}(\{\nu\}, \cdot)$ as our definition of $F_\nu$. This allows us to give a proof that $F_{-}$ is a tensor functor which is unique to $\mathbb{F}_p$-sheaves and simpler than that in \cite[6.4]{GeometricSatake}. In particular, we need only globalize the points $\{\nu\}$ relative to a curve instead of the $S_\nu$. In Subsection \ref{convsection} we globalize the $S_\nu$ to give a proof of the compatibility between convolution and the constant term functor $\CT_L^G$ with respect to a general Levi subgroup $L \subset G$. By taking $L=T$ this provides an alternative proof of Theorem \ref{tensorfunctor} below which is analogous to that in \cite[6.4]{GeometricSatake}.
\end{rmrk}

For $\nu \in X_*(T)_{-}$ let $\{\nu\}(X^2) \subset \Gr_{G,X^2}$ be the reduced closure of $$\bigcup_{\substack{\nu_1, \nu_2 \in X_*(T)_{-} \\ \nu_1 + \nu_2 = \nu}} \{\nu_1\} \times \{\nu_2\} \times U. $$ The reduced fiber of $\{\nu\}(X^2)$ over $\Delta(X)$ is isomorphic to $\{\nu\} \times X \subset \Gr_G \times X$. Denote by $i_{\nu, X^2} \colon \{\nu\}(X^2) \rightarrow \Gr_{G,X^2}$ the inclusion. For $\nu \in X_*(T)_{-}$ and $\mathcal{F}^\bullet \in D_c^b(\Gr_{G,X^2}, \mathbb{F}_p)$ set $$\tilde{F}_\nu (\mathcal{F}^\bullet) : = Rf_{I,!} (Ri_{\nu, X^2, *} (Ri_{\nu, X^2}^* \mathcal{F}^\bullet)) \in D_c^b(X^2, \mathbb{F}_p).$$ 

\begin{thm} \label{tensorfunctor}
The total weight functor is a tensor functor
$$
\xymatrix{
F_{-} \colon (P_{L^+G}(\Gr_G, \mathbb{F}_p),*) \ar[r] &  (\Vect_{\mathbb{F}_p}(X_*(T)_{-}),\otimes).
}
$$  
\end{thm}

\begin{proof}
By the same considerations as in the proof of (\ref{eq1}) in \cite[7.10]{modpGr}, we have
\begin{equation} \label{eq2}
\restr{H^{2\rho(\nu)-2}(\tilde{F}_\nu(Rm_{I,!}\mathcal{F}^\bullet_{1,2}))}{(x_1,x_2)} \cong \left\{
        \begin{array}{ll}
            \bigoplus_{\nu_1 + \nu_2 = \nu}F_{\nu_1}(\mathcal{F}^\bullet_1) \otimes F_{\nu_2}(\mathcal{F}^\bullet_2) & \quad \textrm{if $x_1 \neq x_2$} \\
            F_\nu(\mathcal{F}_1^\bullet * \mathcal{F}_2^\bullet) & \quad \textrm{if $x_1 = x_2$}.
        \end{array}
    \right.
\end{equation}
From the adjunction between $Ri_{\nu, X^2}^*$ and $Ri_{\nu, X^2, *}$ we get a natural map 
\begin{equation} \label{eq3} H^{n-2}(Rf_{I,!}(Rm_{I,!}\mathcal{F}^\bullet_{1,2})) \rightarrow \bigoplus_{2 \rho(\nu) = n} H^{n-2}(\tilde{F}_\nu(Rm_{I,!}\mathcal{F}^\bullet_{1,2})).\end{equation} 
By \ref{Zwt} and the description of the stalks in (\ref{eq1}), (\ref{eq2}) the above map (\ref{eq3}) is an isomorphism over closed points in $X^2$. Since each of the sheaves in (\ref{eq3}) is constructible then this is an isomorphism of sheaves on $X^2$. As $H^{n-2}(Rf_{I,!}(Rm_{I,!}\mathcal{F}^\bullet_{1,2}))$ is constant by \cite[7.9]{modpGr}, then each of the sheaves $H^{n-2}(\tilde{F}_\nu(Rm_{I,!}\mathcal{F}^\bullet_{1,2}))$ is also constant. Hence by (\ref{eq2}) we get a natural isomorphism $$F_\nu(\mathcal{F}_1^\bullet * \mathcal{F}_2^\bullet) \cong \bigoplus_{\nu_1 + \nu_2 = \nu} F_{\nu_1}(\mathcal{F}^\bullet_1) \otimes F_{\nu_2}(\mathcal{F}^\bullet_2).$$ By summing over $\nu \in X_*(T)_-$ we get an isomorphism
$$F_-(\mathcal{F}_1^\bullet * \mathcal{F}_2^\bullet) \cong F_-(\mathcal{F}_1^\bullet) \otimes F_-(\mathcal{F}_2^\bullet).$$

The associativity isomorphism in $P_{L^+G}(\Gr_G, \mathbb{F}_p)$ is constructed from the associativity of the operation $\boxtimes$ (see the proof of \cite[7.11]{modpGr}), so the above isomorphism is compatible with the usual associativity isomorphism in $\Vect_{\mathbb{F}_p}(X_*(T)_{-})$. Moreover, using (\ref{C78}) and (\ref{eq2}) one can verify directly from the construction in \cite[7.11]{modpGr} that the commuta\-ti\-vi\-ty isomorphism in $P_{L^+G}(\Gr_G, \mathbb{F}_p)$ is compatible with the commutativity isomorphism in $\Vect_{\mathbb{F}_p}(X_*(T)_{-})$. Thus $F_{-}$ is a tensor functor.
\end{proof}

We denote by $P_{L^+G}(\Gr_{G},\bbF_p)^{\sss}$
the full subcategory of $P_{L^+G}(\Gr_{G},\bbF_p)$ consisting of semi-simple objects. By \cite[1.2]{modpGr} it is a Tannakian subcategory with fiber functor given by the restriction of $H$.

\begin{cor}\label{F-ss}
The functor
$$
\xymatrix{
F_-|_{(P_{L^+G}(\Gr_G,\bbF_p)^{\sss},*)}:(P_{L^+G}(\Gr_G,\bbF_p)^{\sss},*)\ar[r] & (\Vect_{\bbF_p}(X_*(T)_-),\otimes)
}
$$
is an equivalence of symmetric monoidal categories. We have
$$
\forall\lambda\in X_*(T)^+,\quad F_-(\IC_{\lambda})= \bbF_p(w_0(\lambda)).
$$
\end{cor}

\begin{rmrk} \label{summarizeT}
We can summarize this section as follows. Let $2 \rho_{-}:X_*(T)_{-}\ra\mathbb{Z}_{-}$ be the additive map induced by the group homomorphism $2 \rho:X_*(T)\ra\mathbb{Z}$, and let $2 \rho_{-}:\Vect_{\mathbb{F}_p}(X_*(T)_{-})\ra\Vect_{\mathbb{F}_p}(\mathbb{Z}_{-})$ be the induced functor. Then the exact faithful symmetric monoidal functor 
$$
\xymatrix{
H:(P_{L^+G}(\Gr_G, \mathbb{F}_p),*)\ar[r] & (\Vect_{\mathbb{F}_p},\otimes)
}
$$
factors as a composition of exact faithful symmetric monoidal functors
$$
(P_{L^+G}(\Gr_G, \mathbb{F}_p),*) \xrightarrow{F_-} (\Vect_{\mathbb{F}_p}(X_*(T)_{-}),\otimes) \xrightarrow{2 \rho_{-}} (\Vect_{\mathbb{F}_p}(\mathbb{Z}_{-}),\otimes) \xrightarrow{\text{Forget}} (\Vect_{\mathbb{F}_p},\otimes).
$$
\end{rmrk}

\section{The constant term functor} \label{section_reslevi}

\subsection{The definition of $\CT_L^G$} We return to the setup in Subsection \ref{Levisec} following the geometric setting explained in \cite[\S 5.3.27]{BD}; see also \cite[\S 1.15.1]{BR18}. In particular, $P\subset G$ is a parabolic subgroup containing $B$, and $L\subset P$ is the Levi factor containing $T$. We may consider for $L$ all the objects that we consider for $G$; we will denote them using a letter $L$ as a subscript or a superscript. There is a diagram
\begin{equation} \label{LPGdiagram}
\xymatrix{
&\Gr_P \ar[dl]_{q} \ar[dr]^{p} & \\
\Gr_L & & \Gr_G.
}
\end{equation}
The connected components of $\Gr_L$ are parametrized by
$$
\pi_0(\Gr_L)=\pi_1(L)=X_*(T)/\bbZ\Phi_L^{\vee},
$$
where $\Phi_L^{\vee}$ is the set of coroots of $L$ with respect to $T$. For $c\in \pi_0(\Gr_L)$ let $\Gr_L^c$ and $\Gr_P^c$ be the corresponding connected components of $\Gr_L$ and $\Gr_P$. 

Let $\rho_L$ be half the sum of the positive roots of $L$. Then $2(\rho-\rho_L)(c)$ is a well-defined integer for $c\in \pi_0(\Gr_L)$ since
$\rho=\rho_L$ on $\Phi_L^{\vee}$. Define the locally constant function
\begin{equation} \label{degPdef}
\deg_P \colon \Gr_P \rightarrow \pi_0(\Gr_P) \xrightarrow{2(\rho-\rho_L)}  \mathbb{Z}
\end{equation} where $\Gr_P \rightarrow \pi_0(\Gr_P)$ sends $\Gr_P^c$ to $c$.

\begin{defn} \label{CTdef}
The \emph{$L$-constant term functor} is
$$\xymatrix{
\CT_L^G := Rq_! \circ Rp^* [\deg_P] \colon P_{L^+G}(\Gr_G, \mathbb{F}_p) \rightarrow D_c^b(\Gr_L, \mathbb{F}_p).
}$$
\end{defn}

Let $c\in \pi_0(\Gr_L)$. Since
$$(\Gr_P^c)_{\red} = S_c,$$ then by restricting (\ref{LPGdiagram}) to $S_c$ we get a diagram
$$
\xymatrix{
&S_c \ar[dl]_{\sigma_c} \ar[dr]^{s_c} & \\
\Gr_L^c & & \Gr_G.
}
$$

\begin{defn}
The \emph{weight functor associated to $c$} is
$$
\xymatrix{
F_{c}:=R\sigma_{c!} \circ Rs_{c}^*[2(\rho-\rho_L)(c)] \colon P_{L^+G}(\Gr_G,\bbF_p) \ar[r] & D_c^b(\Gr_L^c,\bbF_p).
}
$$
\end{defn}

\begin{lem} \label{CTweight}
There is a natural isomorphism of functors
$$
\CT_L^G  \cong \bigoplus_{c \in \pi_0(\Gr_L)} F_c.$$
\end{lem}

\begin{proof}
This follows from the definitions and the topological invariance of the \'{e}tale site.
\end{proof}

\subsection{Preservation of perversity}

\begin{thm} \label{CTperv}
Let $c \in \pi_0(\Gr_L)$ and $\cF^{\bullet}\in P_{L^+G}(\Gr_G,\bbF_p)$. Then 
$$
F_c(\cF^{\bullet})  \in P_{L^+L}(\Gr_L, \mathbb{F}_p).
$$ Furthermore, for $\lambda \in X_*(T)^+$ we have
$$
F_c(\IC_\lambda)=
\left\{ \begin{array}{ll} 
\IC^L_{w_0^Lw_0(\lambda)} & \quad \textrm{if $c=c(w_0(\lambda))$}, \\
0 & \quad \textrm{otherwise}.
\end{array}
\right.
$$
\end{thm}

\begin{proof}
The description of $F_c(\IC_\lambda)$ follows from \ref{thmcompute} since $\IC_\lambda = \bbF_p[2 \rho(\lambda)]$ supported on $\Gr_G^{\leq \lambda}$ and $\IC^L_{w_0^Lw_0(\lambda)} = \bbF_p[2 \rho_L(w_0^Lw_0(\lambda))]$ supported on $\Gr_L^{\leq w_0^Lw_0(\lambda)}.$ Then the perversity of $F_c(\cF^{\bullet})$ for general $\cF^\bullet$ follows by induction on the length of $\cF^{\bullet}$. For equivariance, we observe that $\cF^{\bullet}$ is $L^+L$-equivariant, and that $S_c$ is $L^+L$-stable and $\sigma_c \colon S_c \rightarrow \Gr_L^c$ is $L^+L$-equivariant. As pullback along a smooth morphism is $t$-exact (up to a shift) for the perverse $t$-structure by \cite[2.15]{modpGr}, then it follows that $F_c(\cF^{\bullet}) \in P_{L^+L}(\Gr_L^c, \mathbb{F}_p)$ by the proper base change theorem (cf. \cite[3.2]{modpGr}).
\end{proof}

\begin{nota}\label{notaGrGM}
Given a subset $A\subset X_*(T)^+$, we set
$$
\Gr_{G,A}:=\varinjlim_{\lambda\in A}\Gr_{\leq \lambda}.
$$
This is an ind-closed subscheme of $\Gr_G$, which is stable under the $L^+G$-action. There is a natural embedding
$$
P_{L^+G}(\Gr_{G,A},\bbF_p)\subset P_{L^+G}(\Gr_G,\bbF_p)
$$
which identifies $P_{L^+G}(\Gr_{G,A},\bbF_p)$ with the full subcategory of $P_{L^+G}(\Gr_G,\bbF_p)$ whose objects are supported on 
$\Gr_{G,A}$. Let 
$$
\overline{A} = \{ \lambda \in X_*(T)^+ \: : \: \lambda \leq \mu \text{ for some } \mu \in A\}.
$$
Then the simple objects in $P_{L^+G}(\Gr_{G,A},\bbF_p)$ are the $\IC_{\lambda}$ for $\lambda \in \overline{A}$. Moreover, if $A\subset X_*(T)^+$ is a \emph{submonoid}, then so is $\oA$ and it follows from \cite[1.2, 6.7]{modpGr} that the full subcategory $P_{L^+G}(\Gr_{G,A},\bbF_p)$ inherits from $P_{L^+G}(\Gr_G,\bbF_p)$ the structure of a symmetric monoidal category.
\end{nota}

\begin{cor} \label{FcImage}
If $\cF^{\bullet}\in P_{L^+G}(\Gr_G,\bbF_p)$ and $c\cap X_*(T)_-=\emptyset$, then $F_c(\cF^{\bullet})=0$. In general,
$$
F_c(\cF^{\bullet})\in P_{L^+L}(\Gr_{L,w_0^LX_*(T)_-},\bbF_p).
$$
\end{cor}

\begin{proof}
If $\cF^{\bullet}$ is simple this follows from \ref{CTperv}. The general case follows by induction on the length of $\cF^{\bullet}$.
%since $P_{L^+G}(\Gr_{G,A},\bbF_p)\subset P_{L^+G}(\Gr_{G},\bbF_p)$ is thick as follows directly from its definition
\end{proof}

\begin{cor}
The $L$-constant term functor is an exact functor
$$\xymatrix{
\CT_L^G \colon P_{L^+G}(\Gr_G, \bbF_p) \ar[r] & P_{L^+L}(\Gr_{L,w_0^LX_*(T)_-}, \bbF_p).
}$$
\end{cor}

\begin{proof}
This follows from \ref{FcImage} and \ref{CTweight}.
\end{proof}

\noindent Note that for $L=T$, we recover the functor $F_-$, i.e.
$$
\xymatrix{
\CT^G_T=F_-:=\bigoplus_{\nu\in X_*(T)_-}F_{\nu} \colon P_{L^+G}(\Gr_G,\bbF_p) \ar[r] & \Vect_{\bbF_p}(X_*(T)_-).
}
$$
In particular, $\CT^L_T=F_-^L$.

\begin{rmrk}
Let us set
$$
\pi_0(\Gr_L)_-:=\{c\in\pi_0(\Gr_L)\ |\ c\cap X_*(T)_-\neq\emptyset \}=\Im\big(X_*(T)_-\ra X_*(T)/\bbZ\Phi_L^{\vee}\big),
$$
which is a submonoid of the abelian group $\pi_0(\Gr_L)$, and 
$$
\Gr_L^-:=\coprod_{c\in\pi_0(\Gr_L)_-}\Gr_L^c.
$$
Then $\pi_0(\Gr_L^-)=\pi_0(\Gr_L)_-$, $\CT^G_L\cong \bigoplus_{c \in \pi_0(\Gr_L)_-} F_c$ and we have the inclusion
$$
\Gr_{L,w_0^LX_*(T)_-}\subset \Gr_L^-.
$$
The latter is an equality for $L=T$, but it is \emph{strict} in general. Indeed, for any $\alpha^{\vee}\in\Phi_L^{\vee}$, we have $\{\alpha^{\vee}\}\in \Gr_L^0\subset \Gr_L^-$, while $\{\alpha^{\vee}\}\notin  \Gr_{L,w_0^LX_*(T)_-}$ in general, e.g. for $L=GL_2\times GL_1\subset G=GL_3$,
$$
\alpha^{\vee}=(1,-1,0)=w_0^L(-1,1,0)\in X_*(T)_{+/L}\setminus w_0^LX_*(T)_-.
$$
\end{rmrk}

\begin{rmrk}
There is a more general version of Theorem \ref{wtequivalence} as follows. Let $c \in \pi_0(\Gr_L)$ and denote by $i_{c} \colon \Gr_L^c \rightarrow \Gr_G$ the inclusion. Then one can show that there is a natural isomorphism of functors $$F_c \cong {}^pH^{2(\rho-\rho_L)(c)} \circ Ri_c^* \colon P_{L^+G}(\Gr_G,\bbF_p) \rightarrow  P_{L^+L}(\Gr_L^c,\bbF_p).$$ We will only use the functor $F_c$ because it does not require a perverse truncation.
\end{rmrk}

\subsection{Relation to the Satake equivalence}

\begin{prop} \label{transitivitynuc}
Let $c\in \pi_0(\Gr_L)$ and $\nu\in X_*(T)$.
If $\nu\notin c$, then
$$
F_{\nu}^L\circ F_c=0,
$$
and if $\nu\in c$ then
$$
F_{\nu}^L\circ F_c\cong F_{\nu}.
$$
\end{prop}

\begin{proof}
If $\nu\notin c$, then $S_{\nu}\cap \Gr_L^c=\emptyset$ so that $F_{\nu}^L\circ F_c=0$. If $\nu\in c$, then up to possible non-reducedness of the fiber product we have a Cartesian diagram
$$
\xymatrix{
S_{\nu} \ar[r] \ar[d]_{\sigma_{\nu,c}:=} & S_c \ar[d]^{\sigma_c} \\
S_{\nu}^L \ar[r] & \Gr_L^c.
}
$$ 
Hence by the proper base change theorem $(R\sigma_{c!} (Rs_c^* \cF^{\bullet}))|_{S_{\nu}^L } \cong R\sigma_{\nu,c!}(\cF^{\bullet}|_{S_{\nu}})$, so that
$$ \label{basechange}
R\Gamma_c(S_{\nu}^L,F_c(\cF^{\bullet}))\cong R\Gamma_c(S_{\nu},\cF^{\bullet})[2(\rho-\rho_L)(\nu)].
$$ Now take the cohomology of both sides in degree $2\rho_L(\nu)$ .
\end{proof}

\begin{cor} \label{transivityres}
For all $\nu\in X_*(T)$, 
$$
F_{\nu}^L\circ \CT^G_L \cong F_{\nu}.
$$
In particular, there is a canonical transitivity isomorphism
$$
H^L\circ \CT^G_L\cong H,
$$
and the functor $\CT^G_L$ is faithful.
\end{cor}

\begin{proof}
The first part follows from \ref{CTweight} and  \ref{transitivitynuc}. Then the transitivity isomorphism is obtained by summing over $\nu$ (in $X_*(T)_-$). Finally the faithfulness of $\CT^G_L$ follows from the transitivity isomorphism and the faithfulness of $H$.
\end{proof}

\subsection{The ind-schemes $S_c(X)$ and $S_c(X^2)$} For $c \in \pi_0(\Gr_L)$ let $S_c(X) \subset \Gr_{G, X}$ and $S_c(X^2) \subset \Gr_{G,X^2}$ be the reduced ind-subschemes realizing relative versions of $S_c$ as in \cite[\S 1.15.1]{BR18}. They can be identified with the corresponding connected components of $(\Gr_{P,X})_{\red}$ and $(\Gr_{P,X^2})_{\red}$. Let $\Gr_{L,X}^{c}$ and $\Gr_{L,X^2}^{c}$ denote the connected components of $\Gr_{L,X}$ and $\Gr_{L,X^2}$ determined by $c$. We denote the relative versions of the ind-immersion $s_c \colon S_c \rightarrow \Gr_G$ and the projection $\sigma_c \colon S_c \rightarrow \Gr_{L}^c$ as follows:
\begin{align*}
\tilde{s}_c & \colon S_c(X) \rightarrow \Gr_{G, X} & \tilde{\sigma}_c  \colon S_c(X) \rightarrow \Gr_{L,X}^c \\
\tilde{s}^2_c & \colon S_c(X^2) \rightarrow \Gr_{G, X^2} & \tilde{\sigma}^2_c  \colon S_c(X^2) \rightarrow \Gr_{L,X^2}^c.
\end{align*}
Since $X= \mathbb{A}^1$ there are canonical isomorphisms 
$$\Gr_{G,X} \cong \Gr_G \times X, \quad \Gr_{L,X} \cong \Gr_L \times X, \quad S_c(X) \cong S_c \times X,$$
in particular we have the projection $\tau \colon \Gr_{G,X}\rightarrow \Gr_G$ and the associated shifted pull-back
$\tau^{\circ} := R\tau^{*}[1] \colon D_c^b(\Gr_G,\mathbb{F}_p) \rightarrow D_c^b(\Gr_{G,X}, \mathbb{F}_p).$

The important facts about the geometry of these ind-schemes are summarized in the following commutative  diagram from \cite[\S 1.15.1]{BR18} whose squares are Cartesian (up to possible non-reducedness of fiber products) and are obtained by restriction to $U \subset X^2$ or its complementary diagonal $\Delta(X) \subset X^2$:

$$\xymatrix{
\restr{(\Gr_{G,X} \times \Gr_{G,X})}{U} \ar[r]^(.6){j_I} & \Gr_{G,X^2} & \Gr_{G,X} \ar[l]_{i_I} \\
\coprod_{c_1+c_2=c} \restr{(S_{c_1}(X) \times S_{c_2}(X))}{U} \ar[u]^{\restr{\tilde{s}^2_c}{U}} \ar[r]^(.70){j_c} \ar[d]_{\restr{\tilde{\sigma}^2_c }{U}} & S_c(X^2) \ar[d]^{\tilde{\sigma}^2_c } \ar[u]^{\tilde{s}^2_c} & \ar[l]_{i_c} S_c(X) \ar[u]^{\tilde{s}_c} \ar[d]_{\tilde{\sigma}_c}  \\
\coprod_{c_1+c_2=c} \restr{(\Gr^{c_1}_{L,X} \times \Gr^{c_2}_{L, X})}{U} \ar[r]^(.70){j_L^c} & \Gr_{L,X^2}^c & \ar[l]_{i_L^c} \Gr^c_{L,X}. 
}$$
We have canonical identifications
$$\tilde{s}_c = s_c \times \id_X \colon S_c \times X \rightarrow \Gr_G \times X, \quad \tilde{\sigma}_c = \sigma_c \times \id_X \colon S_c \times X \rightarrow \Gr_{L}^c \times X,$$ and
$$\restr{\tilde{s}^2_c}{U} = \coprod_{c_1+c_2=c} \restr{(\tilde{s}_{c_1} \times \tilde{s}_{c_2})}{U}, \quad \restr{\tilde{\sigma}^2_c}{U} = \coprod_{c_1+c_2=c}\restr{(\tilde{\sigma}_{c_1} \times \tilde{\sigma}_{c_2})}{U}.
$$

\begin{defn}
Let $c \in \pi_0(\Gr_L)$. Set
$$
\xymatrix{
\tilde{F}_{c}:=R\tilde{\sigma}_{c!} \circ R\tilde{s}_{c}^*[2(\rho-\rho_L)(c)] \colon D_c^b(\Gr_{G,X},\bbF_p) \ar[r] & D_c^b(\Gr_{L,X}^c,\bbF_p),
}
$$
and
$$
\xymatrix{
\tilde{F}^2_{c}:=R\tilde{\sigma}^2_{c!} \circ R\tilde{s}^{2*}_{c}[2(\rho-\rho_L)(c)] \colon D_c^b(\Gr_{G,X^2},\bbF_p) \ar[r] & D_c^b(\Gr_{L,X^2}^c,\bbF_p).
}
$$
\end{defn}

\subsection{The key isomorphism for the compatibility with convolution}

\begin{thm} \label{conviso}
There is a canonical isomorphism
$$\tilde{F}_c^2 \circ j_{I,!*} (\tau^{\circ} \cF_1^\bullet \overset{L}{\boxtimes} \restr{\tau^{\circ} \cF_2^\bullet}{U}) \cong j_{L, !*}^c \left(\bigoplus_{c_1+c_2 =c} \tau_L^\circ F_{c_1}(\cF_1^\bullet) \overset{L}{\boxtimes} \tau_L^\circ \restr{F_{c_2}(\cF_1^\bullet)}{U} \right).$$
\end{thm}

Contrary to the case of characteristic $0$ coefficients, we cannot appeal to Braden's theorem to compute the co-restriction of the left side of  \ref{conviso} over $\Delta(X)$ as in \cite[1.15.2]{BR18}. This complication is the primary obstacle we must overcome in order to prove \ref{conviso}. We begin by reducing to the case where the $\mathcal{F}_i^\bullet$ are simple.

\begin{proof}[Reduction of \ref{conviso} to the case of simple $\mathcal{F}_i^\bullet$]
By a diagram chase involving the proper base change theorem and the K\"unneth formula, the two complexes in \ref{conviso} are canonically identified over $U$. Once we show that the complex on the left is isomorphic to the one on the right, by \cite[2.11]{modpGr} there will be a unique isomorphism which restricts to our canonical identification over $U$.

We claim that it suffices to show the left side is the intermediate extension of its restriction to $U$ in the case where the $\cF_i^\bullet$ are simple. By the properties characterizing $j_{L!*}^c$ in \cite[2.7]{modpGr}, it follows that if the outer two terms in an exact triangle are intermediate extensions, then so is the middle term (cf. the proof of \cite[7.8]{modpGr}). While $j_{I,!*}$ may not be exact in general, (\ref{C78}) allows us to replace $j_{I,!*}$ by the triangulated functor $Rm_{I,!}$. Thus, by induction on the lengths of the $\mathcal{F}_i^\bullet$ we can assume that $\cF_i^\bullet = \IC_{\lambda_i}$ for $\lambda_i \in X_*(T)^+$. \end{proof}

The remainder of the proof will be an explicit computation of both sides of \ref{conviso} in the special case $\cF_i^\bullet = \IC_{\lambda_i}$ for 
$\lambda_i \in X_*(T)^+$. For convenience we denote
$$
\lambda_\bullet:=(\lambda_1,\lambda_2),\quad
|\lambda_\bullet| := \lambda_1 + \lambda_2.
$$

Let $\Gr_{G,X^2}^{\leq\lambda_\bullet}$ be the closure of $\Gr_{G}^{\leq \lambda_1} \times \Gr_G^{\leq \lambda_2} \times {U} \subset \Gr_{G,X^2}$ with its reduced scheme structure. If $p \nmid |\pi_1(G_\text{der})|$ then by \cite[3.1.14]{Z17} we have
\begin{equation} \label{Zhueq} \restr{\Gr_{G,X^2}^{\leq\lambda_\bullet}}{\Delta(X)} \cong \Gr_{G}^{\leq |\lambda_\bullet|} \times X.
\end{equation} If $p \mid |\pi_1(G_\text{der})|$ this isomorphism should be modified by passing to the reduced subscheme on the left side.

\begin{lem} \label{convlem1} There is a canonical isomorphism $$j_{I,!*}(\tau^{\circ} \IC_{\lambda_1} \overset{L}{\boxtimes} \restr{\tau^{\circ} \IC_{\lambda_2}}{U}) \cong \mathbb{F}_p[2\rho(|\lambda_\bullet|)+2] \in P_c^b(\Gr_{G,X^2}^{\leq\lambda_\bullet},\bbF_p).$$
\end{lem}

\begin{proof}
We first observe that $\tau^{\circ} \IC_{\lambda_1} \overset{L}{\boxtimes} \restr{\tau^{\circ} \IC_{\lambda_2}}{U}$ is 
canonically identified with a shifted constant sheaf supported on $\Gr_{G}^{\leq \lambda_1} \times \Gr_G^{\leq \lambda_2} \times {U} \subset \Gr_{G,X^2}$. If $p \nmid |\pi_1(G_\text{der})|$ then $\Gr_{G,X^2}^{\leq\lambda_\bullet}$ is integral and $F$-rational by \cite[7.4]{modpGr}, so $j_{I,!*}(\tau^{\circ} \IC_{\lambda_1} \overset{L}{\boxtimes} \restr{\tau^{\circ} \IC_{\lambda_2}}{U})$ is a shifted constant sheaf supported on $\Gr_{G,X^2}^{\leq\lambda_\bullet}$ by \cite[1.7]{modpGr}. If $p \mid |\pi_1(G_\text{der})|$, choose a $z$-extension $G' \rightarrow G$ and choose lifts $\lambda_1'$, $\lambda_2'$ of $\lambda_1$, $\lambda_2$ to dominant cocharacters of $G'$. The induced morphism $\Gr_{G',X^2}^{\leq\lambda'_\bullet} \rightarrow \Gr_{G,X^2}^{\leq\lambda_\bullet}$ is a universal homeomorphism (see \cite[7.12]{modpGr} for more details), so by topological invariance of the \'{e}tale site it follows that $j_{I,!*}(\tau^{\circ} \IC_{\lambda_1} \overset{L}{\boxtimes} \restr{\tau^{\circ} \IC_{\lambda_2}}{U})$ is still a shifted constant sheaf supported on $\Gr_{G,X^2}^{\leq\lambda_\bullet}$. Hence in any case there is a canonical isomorphism as stated.
%as already explained above when reducing to the case of simple $\mathcal{F}_i^\bullet$
\end{proof}

\begin{lem} \label{convlem2} If $\mathcal{F}_i^\bullet = \IC_{\lambda_i}$ for $\lambda_i \in X_*(T)^+$ and $w_0(|\lambda_\bullet|) \notin c$ then both sides of \ref{conviso} are zero.
\end{lem}

\begin{proof} By the assumption of the lemma, if $c_1+c_2=c$ then $w_0(\lambda_i) \notin c_i$ for $i =1$ or $2$. For such $i$ we have $F_{c_i}(\IC_{\lambda_i}) = 0$ by \ref{CTperv}, so both sides of \ref{conviso} vanish over $U$. 
%on the right by what we have said, and hence also on the left since both sides identify (canonically) on $U$
Therefore the right side of \ref{conviso} vanishes. On the other hand, by \ref{convlem1} and the proper base change theorem,
$$\restr{\tilde{F}_c^2(j_{I,!*}(\tau^{\circ} \IC_{\lambda_1} \overset{L}{\boxtimes} \restr{\tau^{\circ} \IC_{\lambda_2}}{U}))}{\Delta(X)} \cong \tau^{\circ} F_c(\IC_{|\lambda_\bullet|}).$$ This complex is also zero by \ref{CTperv}, so the left side of \ref{conviso} is zero.
\end{proof}

\begin{lem} \label{convlem3} If $\mathcal{F}_i^\bullet = \IC_{\lambda_i}$ for $\lambda_i \in X_*(T)^+$ and $w_0(|\lambda_\bullet|) \in c$, then the right side of \ref{conviso} is canonically isomorphic to the shifted constant sheaf $$\mathbb{F}_p[2 \rho_L(w_0^Lw_0(|\lambda_\bullet|)) +2] \in P_c^b(\Gr_{L,X^2}^{\leq w_0^Lw_0 (\lambda_{\bullet})}, \mathbb{F}_p).$$
\end{lem}

\begin{proof} By \ref{CTperv}, the right side of \ref{conviso} is canonically isomorphic to
$$j_{L, !*}^c \left( \tau_L^\circ \IC^L_{w_0^Lw_0(\lambda_1)} \overset{L}{\boxtimes} \tau_L^\circ \restr{\IC^L_{w_0^Lw_0(\lambda_2)}}{U}\right).$$ Now apply \ref{convlem1} to $L$ instead of $G$.
\end{proof}

From now on we assume $w_0(|\lambda_\bullet|) \in c$. Let $$V^c_{\lambda_\bullet} := \coprod_{\substack{c_1+c_2=c \\ w_0(\lambda_i) \notin c_i \text{ for some } i}} (S_{c_1} \cap \Gr_{G}^{\leq \lambda_1}) \times (S_{c_2} \cap \Gr_{G}^{\leq \lambda_2}) \times U.$$ Then $V^c_{\lambda_\bullet}$ is an open subscheme of $(S_c(X^2) \cap \Gr_{G,X^2}^{\leq\lambda_\bullet})_{\red}$. Let $Z_{\lambda_\bullet}^c \subset S_c(X^2) \cap \Gr_{G,X^2}^{\leq\lambda_\bullet}$ be its complement with the reduced scheme structure. Then $Z_{\lambda_\bullet}^c$ is a locally closed subscheme of $\Gr_{G,X^2}^{\leq\lambda_\bullet}$ such that
$$\restr{Z_{\lambda_\bullet}^c}{U} \cong (S_{c(w_0(\lambda_1))} \cap \Gr_{G}^{\leq \lambda_1}) \times (S_{c(w_0(\lambda_2))} \cap \Gr_{G}^{\leq \lambda_2}) \times U, \quad \left(\restr{Z_{\lambda_\bullet}^c}{\Delta(X)}\right)_{\text{red}} \cong (S_{c} \cap \Gr_{G}^{\leq |\lambda_\bullet|}) \times X.$$ 
By \ref{Lintersection} (ii), $\tilde{\sigma}_c^2$ restricts to a morphism
$$ \xymatrix{
\tilde{\sigma}_{c, \lambda_\bullet}^2 := \restr{\tilde{\sigma}_{c}^2}{Z_{\lambda_\bullet}^c} \colon Z_{\lambda_\bullet}^c \ar[r] & \Gr_{L,X^2}^{\leq w_0^Lw_0 (\lambda_{\bullet})}.}$$

\begin{lem} \label{univhom}
The morphism $\tilde{\sigma}_{c, \lambda_\bullet}^2 \colon Z_{\lambda_\bullet}^c \rightarrow  \Gr_{L,X^2}^{\leq w_0^Lw_0 (\lambda_{\bullet})}$ is a universal homeomorphism.
\end{lem}

\begin{proof}
By \ref{closedSc}, $\tilde{\sigma}_{c, \lambda_\bullet}^2$ restricts to a universal homeomorphism over $U$ and $\Delta(X)$, so it is universally bijective. The natural morphism $(\Gr_{L,X^2}^c)_{\red} \rightarrow S_c(X^2)$ coming from the morphism $L \rightarrow P$ induces a section to $\tilde{\sigma}_{c, \lambda_\bullet}^2$, so it is a universal homeomorphism.
\end{proof}

\begin{lem} \label{convlem4} If $\mathcal{F}_i^\bullet = \IC_{\lambda_i}$ for $\lambda_i \in X_*(T)^+$ and $w_0(|\lambda_\bullet|) \in c$, then the left side of \ref{conviso} is canonically isomorphic to the shifted constant sheaf $$\mathbb{F}_p[2 \rho_L(w_0^Lw_0(|\lambda_\bullet|)) +2] \in P_c^b(\Gr_{L,X^2}^{\leq w_0^Lw_0 (\lambda_{\bullet})}, \mathbb{F}_p).$$
\end{lem}

\begin{proof} By abuse of notation, let us view $\tilde{\sigma}_{c}^2$ as a morphism
$$\xymatrix{S_c(X^2) \cap \Gr_{G,X^2}^{\leq\lambda_\bullet} \ar[r]^(.6){\tilde{\sigma}_{c}^2} & \Gr_{L,X^2}^c.}$$ Then by the definition of $\tilde{F}_c^2$ and \ref{convlem1}, the left side of \ref{conviso} is $$R\tilde{\sigma}_{c !}^2 (\mathbb{F}_p)[2\rho_L(w_0^Lw_0(|\lambda_\bullet|))+2].$$ Let $j_{V^c_{\lambda_\bullet}} \colon V^c_{\lambda_\bullet} \rightarrow S_c(X^2) \cap \Gr_{G,X^2}^{\leq\lambda_\bullet}$ be the inclusion. By \ref{Lintersection} (ii), we have
$$\tilde{\sigma}_{c}^2(V^c_{\lambda_\bullet}) \cap  \Gr_{L,X^2}^{\leq w_0^Lw_0(\lambda_{\bullet})} = \emptyset.$$ 
Now $R\tilde{\sigma}_{c !}^2 \circ R(j_{V^c_{\lambda_\bullet}})_! (\mathbb{F}_p[2\rho_L(w_0^Lw_0(|\lambda_\bullet|))+2])$ is a direct summand of the restriction over $U$ of $R\tilde{\sigma}_{c !}^2 (\mathbb{F}_p)[2\rho_L(w_0^Lw_0(|\lambda_\bullet|))+2]$, hence is supported in $\Gr_{L,X^2}^{\leq w_0^Lw_0(\lambda_{\bullet})}$ by \ref{convlem3} since the left and right sides of \ref{conviso} agree over $U$. It follows that
$$
R(\tilde{\sigma}_{c}^2 \circ j_{V^c_{\lambda_\bullet}})_! (\mathbb{F}_p[2\rho_L(w_0^Lw_0(|\lambda_\bullet|))+2]) = R\tilde{\sigma}_{c !}^2 \circ R(j_{V^c_{\lambda_\bullet}})_! (\mathbb{F}_p[2\rho_L(w_0^Lw_0(|\lambda_\bullet|))+2]) = 0.
$$ 
%Because the left and right sides of \ref{conviso} agree over $U$, then by \ref{convlem3} we have 
%$$
%R(\tilde{\sigma}_{c}^2 \circ j_{V^c_{\lambda_\bullet}})_! (\mathbb{F}_p[2\rho_L(w_0^Lw_0(|\lambda_\bullet|))+2]) = R\tilde{\sigma}_{c !}^2 \circ %R(j_{V^c_{\lambda_\bullet}})_! (\mathbb{F}_p[2\rho_L(w_0^Lw_0(|\lambda_\bullet|))+2]) =0.
%$$ 
Consequently, by applying $R\tilde{\sigma}_{c !}^2$ to the exact triangle associated to the decomposition of $S_c(X^2) \cap \Gr_{G,X^2}^{\leq\lambda_\bullet}$ into $V^c_{\lambda_\bullet}$ and $Z_{\lambda_\bullet}^c$, the left side of \ref{conviso} is
$$R(\tilde{\sigma}_{c, \lambda_\bullet}^2)_! (\mathbb{F}_p)[2 \rho_L(w_0^Lw_0(|\lambda_\bullet|)) +2].$$ Now we conclude by using \ref{univhom}. 
\end{proof}

\begin{proof}[Proof of \ref{conviso}]
We have reduced to the case where $\cF_i^\bullet = \IC_{\lambda_i}$ for $\lambda_i \in X_*(T)^+$. Then if $w_0 (|\lambda_\bullet|) \notin c$ both sides of \ref{conviso} vanish by \ref{convlem2}, and if $w_0 (|\lambda_\bullet|) \in c$ both sides are canonically identified with the same complex
$$\mathbb{F}_p[2 \rho_L(w_0^Lw_0(|\lambda_\bullet|)) +2] \in P_c^b(\Gr_{L,X^2}^{\leq w_0^Lw_0 (\lambda_{\bullet})}, \mathbb{F}_p)$$
by \ref{convlem3} and \ref{convlem4}.
%alternatively we could remove the `canonically' and invoke \cite[2.11]{modpGr} 
\end{proof}

\subsection{Compatibility with convolution} \label{convsection}

\begin{thm}\label{RGLtensor}
The $L$-constant term functor is a tensor functor
$$
\xymatrix{
\CT^G_L\colon (P_{L^+G}(\Gr_G,\bbF_p),*) \ar[r] & (P_{L^+L}(\Gr_{L,w_0^LX_*(T)_-},\bbF_p),*).
}
$$
\end{thm}

\begin{proof}
Let $\mathcal{F}_1^\bullet$, $\mathcal{F}_2^\bullet \in P_{L^+G}(\Gr_G, \mathbb{F}_p)$. Recall from \ref{conv1} the canonical isomorphism
$$
\tau^{\circ} (\cF_1^\bullet * \cF_2^\bullet) \cong Ri_I^* \circ j_{I,!*} (\tau^{\circ} \cF_1^\bullet \overset{L}{\boxtimes} \restr{\tau^{\circ} \cF_2^\bullet}{U})[-1].
$$

Let $c\in\pi_0(\Gr_L)$. First, apply $\tilde{F}_c$. After unwinding the definitions and using the proper base change theorem, there is a canonical isomorphism
$$ %\begin{equation} \label{BR15.3} 
\tilde{F}_c(\tau^{\circ} (\cF_1^\bullet * \cF_2^\bullet)) \cong \tau_L^{\circ}(F_c(\cF_1^\bullet * \cF_2^\bullet)).
$$ %\end{equation}
A similar diagram chase yields a canonical isomorphism
$$ %\begin{equation} \label{BR15.3'} 
\tilde{F}_c(Ri_I^* \circ j_{I,!*} (\tau^{\circ} \cF_1^\bullet \overset{L}{\boxtimes} \restr{\tau^{\circ} \cF_2^\bullet}{U})[-1]) \cong Ri_L^{c*}(\tilde{F}_c^2 \circ j_{I,!*} (\tau^{\circ} \cF_1^\bullet \overset{L}{\boxtimes} \restr{\tau^{\circ} \cF_2^\bullet}{U}))[-1].
$$ %\end{equation}
Whence
$$
\tau_L^{\circ}(F_c(\cF_1^\bullet * \cF_2^\bullet))\cong Ri_L^{c*}(\tilde{F}_c^2 \circ j_{I,!*} (\tau^{\circ} \cF_1^\bullet \overset{L}{\boxtimes} \restr{\tau^{\circ} \cF_2^\bullet}{U}))[-1].
$$

Second, use the key isomorphism \ref{conviso} to get
$$
\tau_L^{\circ}(F_c(\cF_1^\bullet * \cF_2^\bullet))\cong Ri_L^{c*}\circ j_{L, !*}^c \left(\bigoplus_{c_1+c_2 =c} \tau_L^\circ F_{c_1}(\cF_1^\bullet) \overset{L}{\boxtimes} \tau_L^\circ \restr{F_{c_2}(\cF_1^\bullet)}{U} \right)[-1].
$$

Third, use \ref{conv1} for $L$ instead of $G$ to get
$$
\tau_L^{\circ}(F_c(\cF_1^\bullet * \cF_2^\bullet))\cong \bigoplus_{c_1+c_2=c} \tau_L^{\circ}(F_{c_1}(\mathcal{F}_1^\bullet) * F_{c_2}(\mathcal{F}_2^\bullet)).
$$

By taking the sum over the $c\in \pi_0(\Gr_L)$ we obtain finally (cf. \ref{CTweight})
$$
\CT^G_L( \mathcal{F}_1^\bullet * \mathcal{F}_2^\bullet) \cong \CT_L^G(\mathcal{F}_1^\bullet) * \CT_L^G(\mathcal{F}_2^\bullet).
$$ 

By appealing to the constructions in Subsection \ref{recallconv} one can verify that this isomorphism is compatible with the associativity and commutativity constraints. The arguments are analogous to the case of characteristic $0$ coefficients as in \cite[1.15.2]{BR18}; we leave the details to the reader.
\end{proof}

\begin{cor}\label{Rss}
The functor $\CT^G_L$ induces an equivalence of symmetric monoidal categories
$$
\xymatrix{
\CT^G_L|_{(P_{L^+G}(\Gr_G,\bbF_p)^{\sss},*)}:(P_{L^+G}(\Gr_G,\bbF_p)^{\sss},*)\ar[r]^<<<<<{\sim} & (P_{L^+L}(\Gr_{L,w_0^LX_*(T)_-},\bbF_p)^{\sss},*).
}
$$
We have
$$
\forall\lambda\in X_*(T)^+,\quad \CT^G_L(\IC_{\lambda}) = \IC^L_{w_0^Lw_0(\lambda)}.
$$
\end{cor}

\begin{proof} 
The last assertion follows from \ref{CTweight} and \ref{CTperv}. In particular, it implies that the restriction
$\CT^G_L|_{P_{L^+G}(\Gr_G,\bbF_p)^{\sss}}$ factors through 
$$
P_{L^+L}(\Gr_{L,w_0^LX_*(T)_-},\bbF_p)^{\sss}\subset P_{L^+L}(\Gr_{L,w_0^LX_*(T)_-},\bbF_p).
$$
 Combined with \cite[1.2]{modpGr}, it also implies that $\CT^G_L|_{(P_{L^+G}(\Gr_G,\bbF_p)^{\sss},*)}$ is a tensor functor, which is also a consequence of \ref{RGLtensor}. 
 %Moreover, the unit $\IC_0$ is sent to the unit $\IC^L_0$. 

To conclude the proof, it remains to see that $\CT^G_L$ induces a bijection between the sets of (isomorphism classes of) simple objects, in other words, that the inclusion  
$$
w_0^LX_*(T)_-\subset \overline{w_0^LX_*(T)_-} = \{ \lambda \in X_*(T)_{+/L} \: : \: \lambda \leq_L \mu \text{ for some } \mu \in w_0^LX_*(T)_-\}
$$
is an equality. So let $\lambda\in\overline{w_0^LX_*(T)_-}$, and pick $\mu \in w_0^LX_*(T)_-$ such that $\lambda \leq_L \mu$. Set 
$$
\lambda':=w_0^L(\lambda)\in X_*(T)_{-/L}\quad\textrm{and}\quad\mu':=w_0^L(\mu)\in X_*(T)_-.
$$
We need to check that $\lambda'\in X_*(T)_-$, which means that $\langle \alpha,\lambda' \rangle\leq 0$ for all the simple roots 
$\alpha\in\Delta\subset\Phi$. The inequality holds if $\alpha\in \Delta_L\subset\Delta$ since $\lambda'\in X_*(T)_{-/L}$. Now assume that 
$\alpha\in\Delta\setminus \Delta_L$. As $\lambda \leq_L \mu$, we have $\mu' \leq_L \lambda'$ i.e.
$$
\lambda'\in \mu' + \bbN\Delta_L^{\vee}.
$$
Moreover, as $\mu'\in X_*(T)_-$, we have $\langle \alpha,\mu' \rangle\leq 0$. Lastly, if $\beta\in \Delta_L$, then $\beta\in \Delta$ and 
$\beta\neq \alpha$, so that $\alpha$ and $\beta$ are two distinct elements of a root basis, which implies $\langle\alpha,\beta^{\vee}\rangle\leq 0$.
\end{proof}

\begin{rmrk} \label{summarizeL}
We can summarize this section as follows. The exact faithful symmetric monoidal functor 
$$
\xymatrix{
F_-:(P_{L^+G}(\Gr_G, \mathbb{F}_p),*)\ar[r] & (\Vect_{\mathbb{F}_p}(X_*(T)_{-}),\otimes)
}
$$
can be rewritten as 
$$
\xymatrix{
\CT^G_T:(P_{L^+G}(\Gr_G, \mathbb{F}_p),*)\ar[r] & (P_{L^+T}(\Gr_{T,X_*(T)_{-}}, \mathbb{F}_p),*)
}
$$
and factors as a composition of exact faithful symmetric monoidal functors
$$
\xymatrix{
(P_{L^+G}(\Gr_G, \mathbb{F}_p),*) \ar[r]^<<<<<{\CT^G_L} & (P_{L^+L}(\Gr_{L,w_0^LX_*(T)_-},\bbF_p),*) \ar@{}[d]^{\bigcap}\\
 & (P_{L^+L}(\Gr_L,\bbF_p),*) \ar[r]^<<<<<<<<{\CT^L_T} & (P_{L^+T}(\Gr_{T,X_*(T)_{-/L}}, \mathbb{F}_p),*) \ar@{}[d]^{\bigcap} \\
 & & (P_{L^+T}(\Gr_{T}, \mathbb{F}_p),*)
}
$$(with values in $P_{L^+T}(\Gr_{T,X_*(T)_{-}}, \mathbb{F}_p)\subset P_{L^+T}(\Gr_{T}, \mathbb{F}_p)$).
\end{rmrk}

\section{Tannakian interpretation}\label{sectionTannakian}

\subsection{The Satake equivalence}
Recall from \cite[1.1]{modpGr} the Tannaka equivalence given by the geometric Satake equivalence with $\bbF_p$-coefficients: 
$$
\xymatrix{
(P_{L^+G}(\Gr_G,\bbF_p),*)\ar[rr]_<<<<<<<<<<<<<<<<<{\sim}^{\cS_G} \ar[dr]_H && (\Rep_{\bbF_p}(M_G),\otimes) \ar[dl]^{\rm forget} \\
&(\Vect_{\bbF_p},\otimes).&
}
$$
In particular $M_G$ is an affine monoid scheme over $\bbF_p$ which represents the functor of tensor endomorphisms of the fiber functor $H$.

\begin{nota}\label{SGM}

\begin{itemize}
\item Let $A\subset X_*(T)^+$ be a submonoid. Then the full subcategory $P_{L^+G}(\Gr_{G,A},\bbF_p)\subset P_{L^+G}(\Gr_G,\bbF_p)$ introduced in \ref{notaGrGM} is a Tannakian subcategory with fiber functor given by the restriction of $H$. We denote by $M_{G,A}$ the corresponding $\bbF_p$-monoid scheme and by $\cS_{G,A}$ the resulting Tannaka equivalence. It fits into a commutative diagram
$$
\xymatrix{
(P_{L^+G}(\Gr_{G,A},\bbF_p),*)\ar[r]_{\sim}^{\cS_{G,A}} \ar@{}[d]^{\bigcap} & (\Rep_{\bbF_p}(M_{G,A}),\otimes) \ar@{}[d]^{\bigcap} \\
(P_{L^+G}(\Gr_G,\bbF_p),*)\ar[r]_{\sim}^{\cS_G} & (\Rep_{\bbF_p}(M_G),\otimes).
}
$$
We have a canonical homomorphism $M_G\ra M_{G,A}$, which for $A=X_*(T)^+$ is the identity.

\medskip

\item Given an arbitrary abstract abelian monoid $A$, the category $(\Vect_{\bbF_p}(A),\otimes)$ introduced in \ref{notagradsp} is Tannakian with fiber functor given by forgetting the grading. Its Tannaka monoid is the diagonalizable $\bbF_p$-monoid scheme 
$$
D(A):=\Spec(\bbF_p[A]).
$$
\end{itemize}
\end{nota}

\begin{rmrk}
In the case $G=T$, we have 
$$
M_{T,A}=D(A)
$$
for all submonoids $A\subset X_*(T)$. In particular $M_T=M_{T,X_*(T)}=D(X_*(T))=T^{\vee}$, the torus dual to $T$.
\end{rmrk}

\subsection{The dual of the torus embedding}\label{dualtorusemb}
As noticed in \ref{summarizeT}, we have obtained a factori\-zation of $H$ as
$$
(P_{L^+G}(\Gr_G, \mathbb{F}_p),*) \xrightarrow{F_-} (\Vect_{\mathbb{F}_p}(X_*(T)_{-}),\otimes) \xrightarrow{2 \rho_{-}} (\Vect_{\mathbb{F}_p}(\mathbb{Z}_{-}),\otimes) \xrightarrow{\text{Forget}} (\Vect_{\mathbb{F}_p},\otimes).
$$

Under the equivalences $\cS_G$ and $\cS_T$ it corresponds to a sequence of tensor functors
$$
(\Rep_{\bbF_p}(M_G),\otimes)  \lra (\Rep_{\bbF_p}(D(X_*(T)_{-})),\otimes)\lra (\Rep_{\bbF_p}(\bbA^1_- ),\otimes) \lra(\Rep_{\bbF_p}(1_{\bbF_p}),\otimes),
$$
i.e. by Tannaka duality to a sequence of morphisms of  $\bbF_p$-monoid schemes
$$
\xymatrix{
1_{\bbF_p} \ar[r] & \bbA^1_-  \ar[r]^<<<<<{2\rho_-} & D(X_*(T)_{-}) \ar[r]^<<<<<{D(F_-)} &  M_G \\
 & \bbG_m \ar[r]^{2\rho} \ar@{^{(}->}[u] & T^{\vee}. \ar@{^{(}->}[u] 
}
$$

\begin{rmrk} \label{openclosed}
We show in \ref{propsoffw} that $D(F_-)$, denoted there by $\fw$, is a closed immersion, and that $T^\vee \rightarrow D(X_*(T)_-)$ is an open immersion. 
\end{rmrk}

\subsection{The dual of the Levi embedding}
As noticed in \ref{summarizeL}, we have obtained a factori\-zation of $F_-$ as
$$
(P_{L^+G}(\Gr_G, \mathbb{F}_p),*) \xrightarrow{\CT^G_L} (P_{L^+L}(\Gr_{L,w_0^LX_*(T)_-},\bbF_p),*) 
$$
$$
\subset
(P_{L^+L}(\Gr_L,\bbF_p),*) \xrightarrow{\CT^L_T=F^L_-} (P_{L^+T}(\Gr_{T,X_*(T)_{-/L}}, \mathbb{F}_p),*) 
\subset
(P_{L^+T}(\Gr_{T}, \mathbb{F}_p),*).
$$

Under the equivalences $\cS_G$, $\cS_L$, and $\cS_T$ it corresponds to a diagram
$$
(\Rep_{\bbF_p}(M_G),\otimes) 
%\xrightarrow{R^G_L} 
\lra (\Rep_{\bbF_p}(M_{L, w_0^L X_*(T)_-}),\otimes) 
$$
$$
\subset
(\Rep_{\bbF_p}(M_L),\otimes) 
%\xrightarrow{R^L_T} 
\lra (\Rep_{\bbF_p}(M_{T,X_*(T)_{-/L}}),\otimes) 
\subset
(\Rep_{\bbF_p}(T^{\vee}),\otimes),
$$
i.e. by Tannaka duality to a sequence of morphisms of  $\bbF_p$-monoid schemes
$$
\xymatrix{
T^{\vee} \ar[r] \ar@{^{(}->}[d] & M_{T,X_*(T)_{-/L}} \ar[r]^<<<{D(F_-^L)} & M_L \ar[r] & M_{L,w_0^LX_*(T)_-} \ar[r]^<<<<{D(\CT^G_L)} & M_G. \\
M_{T,X_*(T)_{-}} \ar@/_1pc/[urrrr]_{D(F_-)}
}
$$

\begin{rmrk} \label{futureext}
The morphisms $T^{\vee} \rightarrow M_{T,X_*(T)_{-/L}}$ and $T^{\vee} \rightarrow M_{T,X_*(T)_{-}}$ are open immersions, and the morphisms $D(F_-^L)$ and $D(F_-)$ are closed immersions (cf. \ref{openclosed}). Deciding whether $M_L \rightarrow  M_{L,w_0^LX_*(T)_-}$ and $D(\CT^G_L)$ are open or closed immersions in general seems to require a greater understanding of the  extensions between representations of $M_G$ (and $M_L$), and how these extensions interact with the constant term functors. 
\end{rmrk}

\subsection{Semi-simplification.}\label{Semi-simplification}

\begin{nota} \label{sssnota}
Let $A\subset X_*(T)^+$ be a submonoid. Similarly as in \cite[proof of 7.14]{modpGr}, we denote by $P_{L^+G}(\Gr_{G,A},\bbF_p)^{\sss}$
the full subcategory of $P_{L^+G}(\Gr_{G,A},\bbF_p)$ consisting of semi-simple objects. As noticed in \ref{notaGrGM}, the simple objects are the $\IC_{\lambda}$ for $\lambda\in\overline{A}$, and hence it follows from \cite[1.2]{modpGr} that $P_{L^+G}(\Gr_{G,A},\bbF_p)^{\sss}$ is a Tannakian subcategory with fiber functor given by the restriction of $H$. Then by \ref{HFtheorem} the corresponding $\bbF_p$-monoid scheme is
$$
M_{G,A}^{\sss}:=D(w_0\overline{A}).
$$
The resulting Tannaka equivalence $\cS_{G,A}^{\sss}$ fits into the commutative diagram
$$
\xymatrix{
(P_{L^+G}(\Gr_G,\bbF_p)^{\sss},*)\ar[r]_{\sim}^{\cS_{G,A}^{\sss}} \ar@{}[d]^{\bigcap} & (\Rep_{\bbF_p}(M_{G,A}^{\sss}),\otimes) \ar@{}[d]^{\bigcap} \\
(P_{L^+G}(\Gr_G,\bbF_p),*)\ar[r]_{\sim}^{\cS_{G,A}} & (\Rep_{\bbF_p}(M_{G,A}),\otimes).
}
$$
We have a canonical homomorphism $\pi_{G,M}:M_{G,A}\ra M_{G,A}^{\sss}$. As $M_{G,X_*(T)^+}=M_G$, we write simply $M_{G}^{\sss}$ for $M_{G,X_*(T)^+}^{\sss}=D(X_*(T)_-)$ and $\pi_G:M_G\ra M_G^{\sss}$ for the corresponding canonical homomorphism.
\end{nota}

\begin{rmrk}
Since every simple object of $\Rep_{\bbF_p}(M_G)\cong P_{L^+G}(\Gr_G,\bbF_p)$ is $1$-dimensional, the monoid $M_G$ is pro-solvable, cf. \cite[7.15]{modpGr}. Let $\{V_{\lambda},\lambda\in X_*(T)_-\}$ be (a set of representatives of) the set of irreducible finite dimensional representations of $M_G$. For any $V\in \Rep_{\bbF_p}(M_G)$, denote by $d_{\lambda}(V)$ the multiplicity of $V_{\lambda}$ as a subquotient in any Jordan-H\"older filtration of $V$. Then the canonical homomorphism $\pi_G:M_G\ra M_{G}^{\sss}$ admits the following explicit description: it maps $m\in M_G$ to the unique
$\pi_G(m)\in M_{G}^{\sss}=D(X_*(T)_-)$ acting  
$$
\textrm{on}\quad\bigoplus_{\lambda\in X_*(T)_-}V_{\lambda}^{d_{\lambda}(V)}
\quad\textrm{by}\quad \sum_{\lambda\in X_*(T)_-}\lambda(\pi_G(m))
$$
for all $V\in \Rep_{\bbF_p}(M_G)$.
\end{rmrk}

\begin{defn}\label{defpiG}
We call the canonical homomorphism 
$$
\xymatrix{
\pi_G:M_G\ar[r] & M_G^{\sss}
}
$$
the \emph{eigenvalues homomorphism}.
\end{defn}
%compare to $B \ra B/U=T$ 

From \ref{F-ss} we have the equivalence
$$
\xymatrix{
F_-|_{(P_{L^+G}(\Gr_G,\bbF_p)^{\sss},*)}:(P_{L^+G}(\Gr_G,\bbF_p)^{\sss},*)\ar[r]^<<<<{\sim} & (\Vect_{\bbF_p}(X_*(T)_-),\otimes),
}
$$
such that $F_-(\IC_{\lambda})=\bbF_p(w_0(\lambda))$. By Tannaka duality, it corresponds to the identity
$$
\xymatrix{
M_{T,X_*(T)_-}=D(X_*(T)_-) \ar@{=}[r] & D(X_*(T)_-)=M_G^{\sss}.
}
$$

\begin{defn} \label{defss}

\begin{itemize}
\item We call the composition
$$
\xymatrix{
(\cdot)^{\sss}:=F_-|_{(P_{L^+G}(\Gr_G,\bbF_p)^{\sss},*)}^{-1}\circ F_-\colon (P_{L^+G}(\Gr_G,\bbF_p),*) \ar[r] & (P_{L^+G}(\Gr_G,\bbF_p)^{\sss},*)
}
$$
the \emph{semi-simplification functor}. 

\medskip

\item We call its Tannaka dual 
$$
\xymatrix{
\fw:=D((\cdot)^{\sss}):M_G^{\sss} \ar[r] & M_G
}
$$
the \emph{weight section}. 
\end{itemize}
\end{defn}

\noindent Thus the functor $(\cdot)^{\sss}$ is a retraction to $P_{L^+G}(\Gr_G,\bbF_p)^{\sss}\subset P_{L^+G}(\Gr_G,\bbF_p)$ and the morphism $\fw$ is a section to $\pi_G:M_G\ra M_G^{\sss}$. Moreover $\fw$ identifies with the morphism $D(F_-)$ from Subsection \ref{dualtorusemb}. 

\begin{thm} \label{propsoffw} The morphisms $\pi_G$ and $\fw$ satisfy the following properties.
\begin{itemize}
\item The morphism $\pi_G \colon M_G \rightarrow M_G^{\sss}$ is surjective.
\item The weight section $\fw:M_G^{\sss}\ra  M_G$ is a closed immersion. The dual torus embedding $T^{\vee}\ra M_G$ factors through $\fw$ by an open immersion. 
\end{itemize}
\end{thm}

\begin{proof}
The weight section $\fw$ of $\pi_G$ is a closed immersion since the morphism $\pi_G$ is affine. Conversely, the fact that $\pi_G$ admits a section implies that $\pi_G$ is surjective. 

By construction, we have the commutative diagram
$$
\xymatrix{
 M_G^{\sss}  \ar@{=}[r] \ar@/^1pc/[rr]^{\fw} & M_{T,X_*(T)_{-}}\ar[r]_<<<<<{D(F_-)} &  M_G \\
 & T^{\vee}. \ar@{^{(}->}[u] 
}
$$
The fact that $T^{\vee}\ra M_{T,X_*(T)_{-}}= M_G^{\sss}$ is an open immersion will be shown in \ref{lemdecomp}.
\end{proof}

From \ref{Rss} we have the equivalence
$$
\xymatrix{
\CT^G_L|_{(P_{L^+G}(\Gr_G,\bbF_p)^{\sss},*)}:(P_{L^+G}(\Gr_G,\bbF_p)^{\sss},*)\ar[r]^<<<<<{\sim} & (P_{L^+L}(\Gr_{L,w_0^LX_*(T)_-},\bbF_p)^{\sss},*),
}
$$
such that $\CT^G_L(\IC_{\lambda})=\IC^L_{w_0^Lw_0(\lambda)}$. By Tannaka duality, it corresponds to the identity 
$$
\xymatrix{
M_{L,w_0^LX_*(T)_-}^{\sss}=D(X_*(T)_-) \ar@{=}[r] & D(X_*(T)_-)=M_G^{\sss}.
}
$$

\section{The space of Satake parameters} \label{section_space}

\subsection{The definition of Satake parameters}
The \emph{space of Satake parameters} is the $\bbF_p$-scheme
$$
\sP:=\Spec(\bbF_p[X_*(T)_-])
$$
underlying the $\bbF_p$-monoid scheme $D(X_*(T)_-)=M_G^{\sss}$. 

A \emph{Satake parameter} is an $\bbF_p$-point of $\sP$.

\begin{defn}
Let $X$ be a scheme. A \emph{stratification} of $X$ is a decomposition $X=\cup_{i\in I} X_i$ \ref{defdecompfil} such that for all $i\in I$, the closure of $X_i$ in $X$ is a union of some $X_j$'s, i.e. there exists $J_i\subset I$ such that 
$$
|\overline{X_i}|=\bigcup_{j\in J_i} |X_j|.
$$
\end{defn}

We are going to define a stratification of the space of Satake parameters by first defining the relevant categories of equivariant perverse sheaves on the affine Grassmannian and then applying Tannaka duality.

\subsection{The closed stratum} Let us set
$$
\Delta^{\perp}:=\{\lambda\in X_*(T)\ |\ \langle \alpha,\lambda \rangle=0\ \forall \alpha\in\Delta\}.
$$
Then for all $\lambda\in\Delta^{\perp}$, we have $\dim\Gr_G^{\leq\lambda}=2\rho(\lambda)=0$, so that $\Gr_G^{\leq\lambda}=\{\lambda\}$ and hence
$$
\Gr_{G,\Delta^{\perp}}=\coprod_{\lambda\in \Delta^{\perp}} \{\lambda\}.
$$ 
Consequently, the embedding $P_{L^+G}(\Gr_{G,\Delta^{\perp}},\bbF_p)\subset P_{L^+G}(\Gr_{G},\bbF_p)$ factors as 
$$
P_{L^+G}(\Gr_{G,\Delta^{\perp}},\bbF_p)\subset P_{L^+G}(\Gr_{G},\bbF_p)^{\sss}\subset P_{L^+G}(\Gr_{G},\bbF_p),
$$
and the equivalence of tensor categories
\begin{eqnarray*}
\Vect(X_*(T)^+)& \xrightarrow{\sim} & P_{L^+G}(\Gr_{G},\bbF_p)^{\sss} \\
\bbF_p(\lambda) & \mapsto & \IC_{\lambda}
\end{eqnarray*}
restricts to an equivalence of tensor categories
\begin{eqnarray*}
\Vect(\Delta^{\perp}) & \xrightarrow{\sim} & P_{L^+G}(\Gr_{G,\Delta^{\perp}},\bbF_p) \\
\bbF_p(\lambda) & \mapsto & \IC_{\lambda}.
\end{eqnarray*}

We define a retraction 
$$
\xymatrix{
P_{L^+G}(\Gr_{G,\Delta^{\perp}},\bbF_p) \ar@{^{(}->}[r] & P_{L^+G}(\Gr_{G},\bbF_p)^{\sss} \ar@/_1pc/[l]_r
}
$$
by the rule
\begin{eqnarray*}
r:P_{L^+G}(\Gr_{G},\bbF_p)^{\sss}& \lra & P_{L^+G}(\Gr_{G,\Delta^{\perp}},\bbF_p)\\
\IC_{\lambda}  & \lmapsto & \left\{ \begin{array}{ll}
\IC_{\lambda} & \textrm{if $\lambda\in\Delta^{\perp}$}\\ 
0 & \textrm{otherwise}.
\end{array} \right.
\end{eqnarray*}

\begin{lem}\label{retraction}
The $\bbF_p$-linear functor $r$ is a tensor functor.
\end{lem}

\begin{proof}
Indeed, for $\lambda,\mu\in X_*(T)^+$, we have $\IC_{\lambda}*\IC_{\mu}=\IC_{\lambda+\mu}$ and
$$
\xymatrix{
(\lambda\in \Delta^{\perp}\ \textrm{and}\ \mu\in \Delta^{\perp})\Longleftrightarrow \lambda+\mu\in \Delta^{\perp}.
}
$$
Moreover $r(\IC_0)=\IC_0$.
%and associativity and commutativity constraints are clear here
\end{proof}

Applying the Satake equivalence $\cS_{G,\Delta^{\perp}}$ from \ref{SGM}, we get a tensor retraction
$$
\xymatrix{
\Rep_{\bbF_p}(M_{G,\Delta^{\perp}})\ar@{^{(}->}[r] & \Rep_{\bbF_p}(M_{G}^{\sss}) \ar@/_1pc/[l]_r,
}
$$
which by Tannaka duality corresponds to a multiplicative section 
$$
\xymatrix{
M_{G}^{\sss}=D(X_*(T)_-) \ar@{->>}[r] & M_{G,\Delta^{\perp}}=D(\Delta^{\perp})\  \ar@/_1pc/[l]_s.
}
$$
In particular 
$$
S_G:=s(D(\Delta^{\perp}))
$$ 
is a closed subsemigroup of $D(X_*(T)_-)$.
%No conflict of notation with these strata $S_G$, $S_L$, and the $S_{\nu}$, $S_c$ above...

\begin{lem} \label{flatmonoid}
Let $A$ be an abstract, right cancellative monoid. Let $B \subset A$ be a subgroup and let $R$ be a ring. Then $R[A]$ is a free $R[B]$-module. In particular, the inclusion of rings $R[B] \subset R[A]$ is flat. 
\end{lem}

\begin{proof}
Because $B$ is a group then the right cosets of $B$ in $A$ give a partition of $A$. Thus, if $\{a_i\}_i$ is a collection of representatives for these cosets then
$$
R[A] = \bigoplus_i R[Ba_i].
$$ 
Since $A$ is right cancellative then each morphism $R[B] \ra R[Ba_i]$, $b \mapsto ba_i$, is an isomorphism. 
\end{proof}

\begin{prop}
The morphism $M_{G}^{\sss} \rightarrow M_{G,\Delta^{\perp}}$ is faithfully flat. 
\end{prop}

\begin{proof}
It is flat by \ref{flatmonoid} applied to the monoid $X_*(T)_-$ and the subgroup $\Delta^{\perp}$. It is surjective since it admits a section, namely $s$.
\end{proof}

\begin{rmrk}
If $G$ is not a torus then the functor $r$ does \emph{not} intertwine the fiber functors $H|_{P_{L^+G}(\Gr_{G},\bbF_p)^{\sss}}$ and $H|_{P_{L^+G}(\Gr_{G,\Delta^{\perp}},\bbF_p)}$. Cor\-respondingly, the section $s$ does not send the unit of the group scheme $D(\Delta^{\perp})$ to the unit of the monoid scheme $D(X_*(T)_-)$.
\end{rmrk}

\subsection{The open complement to the closed stratum} \label{openstratum}
Recall that a standard Levi subgroup of $G$ is the Levi factor containing $T$ of a parabolic subgroup of $G$ containing $B$. We denote by $\cL$  the set of standard Levi subgroups of $G$. It is in 1-1 correspondence with the power set of the set $\Delta$ of simple roots corresponding to the pair $(B,T)$:
\begin{eqnarray*}
\cL & \xrightarrow{\sim} & \cP(\Delta)  \\
L & \mapsto & \Delta_L
\end{eqnarray*}
where $\Delta_L$ is the set of simple roots of $L$ with respect to the pair $(B\cap L,T)$. In particular $\Delta_T=\emptyset$ and $\Delta_G=\Delta$.

For each $L\in\cL$, we have constructed the functor
$$
\xymatrix{
P_{L^+G}(\Gr_G,\bbF_p)^{\sss}\ar[rr]_<<<<<<<<<{\sim}^>>>>>>>>>>{\CT^G_L|_{P_{L^+G}(\Gr_G,\bbF_p)^{\sss}}} && P_{L^+L}(\Gr_{L,w_0^LX_*(T)_-},\bbF_p)^{\sss}\subset P_{L^+L}(\Gr_{L},\bbF_p)^{\sss},
}
$$
which corresponds to
$$
\xymatrix{
j_L:M_{L}^{\sss}=D(X_*(T)_{-/L}) \ar[r] & M_{L,w_0^LX_*(T)_-}^{\sss}=D(X_*(T)_-) \ar@{=}[r] & D(X_*(T)_-)=M_G^{\sss},
}
$$
cf. end of \ref{Semi-simplification}.

\begin{lem}\label{lemdecomp}
\begin{itemize}
\item The morphism of $\bbF_p$-monoid schemes $j_L$ is an open immersion.
\item For all $L,L'\in \cL$, we have
$$
j_L(D(X_*(T)_{-/L}))\cap j_{L'}(D(X_*(T)_{-/L'}))=j_{L''}(D(X_*(T)_{-/L''}))
$$
with $\Delta_{L''}:=\Delta_L\cap\Delta_{L'}$. 
\item We have
$$
\sP\setminus S_G=\bigcup_{L\in\cL\setminus\{G\}}j_L(D(X_*(T)_{-/L})).
$$
\end{itemize}
\end{lem}

\begin{proof}
By construction $j_L^*:\bbF_p[X_*(T)_-]\ra\bbF_p[X_*(T)_{-/L}]$ is the morphism of $\bbF_p$-algebras induced by the canonical inclusion $X_*(T)_-\subset X_*(T)_{-/L}$. Let $\lambda_{\alpha}$, $\alpha\in\Delta$, be elements of $X_*(T)_-$ such that 
$$
\forall \alpha,\beta \in\Delta,\quad \langle\alpha,\lambda_{\beta}\rangle 
\ \left\{ \begin{array}{ll}
\in\bbZ_{\leq- 1} & \textrm{if $\alpha=\beta$}\\ 
=0 & \textrm{otherwise}
\end{array} \right.
$$
(complete $\Delta$ into a basis of $X^*(T)\otimes\bbQ$ and consider the dual basis of $X_*(T)\otimes\bbQ$ under the perfect pairing 
$\langle\ ,\ \rangle$). Then, for all $\lambda\in X_*(T)_{-/L}$, we can find some $n_{\alpha}\in\bbZ_{\geq 0}$, 
$\alpha\in\Delta\setminus\Delta_L$, such that
$$
(\lambda + \sum_{\alpha\in\Delta\setminus\Delta_L}n_{\alpha}\lambda_{\alpha})\in X_*(T)_-,
$$
i.e.
$$
\bbF_p[X_*(T)_{-/L}]=\bbF_p[X_*(T)_-][(e^{\lambda_{\alpha}})^{-1},\alpha\in\Delta\setminus\Delta_L].
$$
Hence $j_L$ is an open immersion, and the complement of $j_L(D(X_*(T)_{-/L}))$ in $\sP=D(X_*(T)_-)$ is the closed subset defined by the equation $\prod_{\alpha\in\Delta\setminus\Delta_L}e^{\lambda_{\alpha}}=0$. 

Consequently,
$$
\sP\setminus  j_L(D(X_*(T)_{-/L}))\cap j_{L'}(D(X_*(T)_{-/L'}))
$$
is the closed subset defined by the equation $\prod_{\alpha\in\Delta\setminus(\Delta_L\cap \Delta_{L'})}e^{\lambda_{\alpha}}=0$, and hence
$$
j_L(D(X_*(T)_{-/L}))\cap j_{L'}(D(X_*(T)_{-/L'}))=j_{L''}(D(X_*(T)_{-/L''}))
$$
with $\Delta_{L''}:=\Delta_L\cap \Delta_{L'}$. 

Finally,
$$
\sP\setminus \bigcup_{L\in\cL\setminus\{G\}}j_L(D(X_*(T)_{-/L}))
$$
is the closed subset defined by the equations 
$$
\forall \alpha\in\Delta,\quad e^{\lambda_{\alpha}}=0.
$$
On the other hand, 
$$
s(S_G)=V(e^{\lambda},\lambda\in X_*(T)_-\setminus\Delta^{\perp})\subset D(X_*(T)_-)=\sP
$$
by construction. We claim that
$$
(e^{\lambda_{\alpha}},\alpha\in\Delta) \subset (e^{\lambda},\lambda\in X_*(T)_-\setminus\Delta^{\perp}) \subset \sqrt{(e^{\lambda_{\alpha}},\alpha\in\Delta)}.
$$
The first inclusion follows from the definition of the elements $\lambda_\alpha$. For the second one, note that for $\lambda \in X_*(T)_-\setminus\Delta^{\perp}$ we can find integers $m>0$, $m_\alpha \geq 0$, such that $m \lambda - \sum_\alpha m_\alpha \lambda_\alpha \in \Delta^{\perp}$. Since the elements $e^{\mu}$ for $\mu \in \Delta^{\perp}$ are units, the second inclusion follows. Hence $\sP\setminus \bigcup_{L\in\cL\setminus\{G\}}j_L(D(X_*(T)_{-/L}))$ is equal to the subset underlying the closed subscheme $s(S_G)$. 
\end{proof}

From now on we will write simply $D(X_*(T)_{-/L})$ for $j_L(D(X_*(T)_{-/L}))\subset D(X_*(T)_-)$. 
%(but we keep the notation $s(D(\Delta^{\perp}))\subset D(X_*(T)_-)$)

\begin{rmrk}\label{sPintegral}
We have seen in the proof of \ref{lemdecomp} that $T^{\vee}=\Spec(\bbF_p[X_*(T)])$ is the open complement in $\sP=D(X_*(T)_-)$ of the Cartier divisor defined by the regular element
$$
\prod_{\alpha\in\Delta}e^{\lambda_{\alpha}}=e^{\sum_{\alpha\in\Delta}\lambda_{\alpha}}\in \bbF_p[X_*(T)_-].
$$
In particular, the scheme $\sP$ is integral.
\end{rmrk}

\begin{exmp}
If $G = \GL_n$ then $X_*(T)_-=\oplus_{i=1}^{n-1}\mathbb{Z}_{\geq 0}\omega_{i-}\oplus\bbZ\omega_{n-}$ where $\omega_{i-}\in\bbZ^n$ has its first $n-i$ entries equal to $0$ and last $i$ entries equal to $1$, so $\sP = \Spec(\mathbb{F}_p[T_1, \ldots, T_{n-1}, T_n^{\pm 1}])$. 
%cf. also \cite[1.6]{H11b}. 
If $G = \SL_2$ then $X_*(T)_-=\mathbb{Z}_{\geq 0}(-\alpha^{\vee})$ where $-\alpha^{\vee}=(-1,1)$, so $\sP = \Spec(\mathbb{F}_p[T])$. In particular, $\sP$ is smooth in both of these examples.
\end{exmp}

\begin{exmp}
In general $\sP$ is not smooth. For example, let $G = \SL_3$. Then $X_*(T) = \{(a,b,c) \in \mathbb{Z}^3 \: : \: a+b+c=0\}$ and the simple roots are $\alpha = (1,-1,0), \beta = (0,1,-1) \in X^*(T) = \mathbb{Z}^3/\mathbb{Z}$. Then $X_*(T) = \mathbb{Z} \alpha^\vee \oplus \mathbb{Z} \beta^\vee$ and $$X_*(T)^+ = \{ a \alpha^\vee + b \beta^\vee \colon 2a \geq b, 2b \geq a \}.$$ The monoid $X_*(T)^+$ is generated by the elements
$$ \alpha^\vee + \beta^\vee, \quad  \alpha^\vee + 2 \beta^\vee, \quad 2 \alpha^\vee +  \beta^\vee.$$ 
By sending the indeterminates $x,y,z$ to the corresponding generators in $\mathbb{F}_p[X_*(T)^+]$, we get a surjection
$$\mathbb{F}_p[x,y,z]/I \twoheadrightarrow \mathbb{F}_p[X_*(T)^+], \quad I = (x^3-yz).$$ Since $I$ is a prime ideal and 
$\mathbb{F}_p[X_*(T)^+]$ is an integral domain of dimension $2$ then this map is an isomorphism. In particular the ring 
$\mathbb{F}_p[X_*(T)^+]$, equivalently the ring $\mathbb{F}_p[X_*(T)_-]$, is not regular.
\end{exmp}

\subsection{The Herzig stratification}  \label{Herzig_sec}

For all $L\in \cL$, set
$$
S_L:=s_L \big(D(\Delta_L^{\perp})\big).
$$

\begin{cor} \label{herzigstratif}
The space of Satake parameters admits the following stratification by subsemigroups:
$$
\sP=\bigcup_{L\in\cL} S_L.
$$
The stratum $S_L$ is isomorphic to a torus of rank equal to 
$$
\textnormal{rank } T - |\Delta_L|=\textnormal{rank } \pi_1(L)=\textnormal{rank } \pi_0(\Gr_L).
$$
The closure of $S_L$ in $\sP$ is
$$
\overline{S_L}=\bigcup_{L'\supset L}S_{L'}.
$$
\end{cor}

\begin{proof}
The decomposition is a consequence of \ref{lemdecomp}. 

Let $L\in \cL$. Since $\Delta_L^{\perp}$ is a subgroup of the finitely generated free abelian group $X_*(T)$ then $\Delta_L^{\perp}$ is also finitely generated and free. Hence $D(\Delta_L^{\perp})$ is a torus, of rank equal to
$$
\text{rank } \Delta_L^{\perp} = \dim_{\mathbb{Q}} (\mathbb{Z} \Delta_L \otimes \mathbb{Q})^{\perp} = \text{rank } T - |\Delta_L|= \text{rank } X_*(T)/\bbZ\Phi_L^{\vee}.
$$

Finally, with the notation of the proof of \ref{lemdecomp}, we have
\begin{eqnarray*}
S_L&:=&\Spec\bigg(\bbF_p[X_*(T)_{-/L}]/(e^{\lambda},\lambda\in X_*(T)_{-/L}\setminus\Delta_L^{\perp})\bigg)\\
&=&\Spec\bigg(\bbF_p[X_*(T)_-][(e^{\lambda_{\alpha}})^{-1},\alpha\in\Delta\setminus\Delta_L]/(e^{\lambda_{\beta}},\beta\in\Delta_L)\bigg)_{\red}.
\end{eqnarray*}
Thus, setting 
$$
f_L:=\prod_{\alpha\in\Delta\setminus\Delta_L}e^{\lambda_{\alpha}}=e^{\sum_{\alpha\in\Delta\setminus\Delta_L}\lambda_{\alpha}}\in \bbF_p[X_*(T)_-]
$$
and
$$
V_L:=\Spec\bigg(\bbF_p[X_*(T)_-]/(e^{\lambda_{\beta}},\beta\in\Delta_L)\bigg)\subset \Spec(\bbF_p[X_*(T)_-])=\sP,
$$
we have
$$
|S_L|=|D(f_L)|\cap |V_L|\subset \sP
$$
and
$$
|V_L|=\bigcup_{L'\supset L} |D(f_{L'})|\cap |V_{L'}|=\bigcup_{L'\supset L} |S_{L'}|.
$$
Now let us show that $|\overline{S_L}|=|V_L|$. Since $|S_L|=|D(f_L)|\cap |V_L|$, it suffices to show that $f_L$ defines a Cartier divisor \emph{after restriction to $V_L$}, i.e. that its image in the ring of functions on $V_L$ is a regular element. So let $a=\sum_{\lambda}a_{\lambda}e^{\lambda}\in\bbF_p[X_*(T)_-]$ such that 
$$
f_La=\sum_{\beta\in\Delta_L}g_{\beta}e^{\lambda_{\beta}}\in(e^{\lambda_{\beta}},\beta\in\Delta_L).
$$
If $a_{\lambda}\neq 0$ then
$$
\sum_{\alpha\in\Delta\setminus\Delta_L}\lambda_{\alpha}+\lambda=\mu+\lambda_{\beta}
$$
for some $\mu\in X_*(T)_-$ and $\beta\in\Delta_L$. The cocharacter
$$
\nu:=\lambda-\lambda_{\beta}=\mu-\sum_{\alpha\in\Delta\setminus\Delta_L}\lambda_{\alpha}
$$
satisfies $\langle \gamma,\nu \rangle\leq 0$ for all $\gamma\in\Delta$, i.e. $\nu\in X_*(T)_-$. Hence $a\in (e^{\lambda_{\beta}},\beta\in\Delta_L)$, as desired. 
\end{proof}

We call the stratification \ref{herzigstratif} the \emph{Herzig stratification}, since it corresponds to the stra\-ti\-fi\-ca\-tion of the set 
$\sP(\overline{\bbF}_p)$ defined in \cite[\S 1.5, \S 2.4]{H11b}.

\begin{defn}\label{ordinaryss}
We call the open stratum
$$
S_{T}=T^{\vee}=D(X_*(T)) \subset \sP
$$
the \emph{ordinary locus}, and the closed stratum
$$
S_{G}=s_G(D(\Delta^{\perp})) \subset \sP
$$
the \emph{supersingular locus}.
\end{defn}

\begin{exmp}\label{HstratifGL2}
For $G=GL_2$ we have $X_*(T)_-=\bbN(0,1)\oplus\bbZ(1,1)$, the space of Satake parameters is
$$
M_{GL_2}^{\sss}=D(X_*(T)_-)=\Spec(\bbF_p[e^{(0,1)},e^{\pm(1,1)}])=\bbA^1\times\bbG_m,
$$
and the Herzig stratification consists only in the ordinary and the supersingular loci
$$
S_{T}\bigcup S_{G}=(\bbG_m\times\bbG_m)\bigcup (\{0\}\times\bbG_m).
$$
\end{exmp}

\begin{exmp}
The supersingular locus $S_G$ is $0$-dimensional if and only if $G$ is semi-simple, in which case it is just one $\bbF_p$-point.
\end{exmp}

\begin{lem}\label{unitsss}
The ordinary locus $T^{\vee}$ is the group of invertible elements of the monoid $M_G^{\sss}$.
\end{lem}

\begin{proof}
Let $s\in M_G^{\sss}(\overline{\bbF}_p)$. Let $L$ be the element of $\cL$ such that $s\in S_L(\overline{\bbF}_p)$. Then, for all $\lambda\in X_*(T)_{-/L}\setminus\Delta_L^{\perp}$,
$$
\lambda(s)=s^*(e^{\lambda})=0\in \overline{\bbF}_p,
$$
i.e. the character $\lambda:D(X_*(T)_{-/L})\ra\bbA^1$ vanishes on $s$. Hence, if $s$ is invertible in $M_G^{\sss}(\overline{\bbF}_p)=D(X_*(T)_-)(\overline{\bbF}_p)$, then $(X_*(T)_{-/L}\setminus\Delta_L^{\perp})\cap X_*(T)_-=\emptyset$, i.e. $X_*(T)_-\subset \Delta_L^{\perp}$, which occurs only if $\Delta_L=\emptyset$, in which case $L=T$.
\end{proof}

\begin{lem}
The supersingular locus $S_G$ is absorbing in the monoid $M_G^{\sss}$.
\end{lem}

\begin{proof}
Let $s\in S_G(\overline{\bbF}_p)$ and $s'\in M_G^{\sss}(\overline{\bbF}_p)$. Then, for all $\lambda\in X_*(T)_-\setminus\Delta^{\perp}$,
$$
(ss')^*(e^{\lambda})=\lambda(ss')=\lambda(s)\lambda(s')=s^*(e^{\lambda})\lambda(s')=0\in \overline{\bbF}_p.
$$
Thus the $\overline{\bbF}_p$-algebra morphism $(ss')^*:\overline{\bbF}_p[X_*(T)_-]\ra \overline{\bbF}_p$ vanishes on the ideal $(e^{\lambda},\lambda\in X_*(T)_-\setminus\Delta^{\perp})$ of $S_G$ in $M_G^{\sss}$, which means precisely that $ss'\in S_G(\overline{\bbF}_p)$.
\end{proof}

\begin{cor}
Let $\pi_G:M_G\ra M_{G}^{\sss}$ be the canonical eigenvalues homomorphism. Then $\pi_G^{-1}(T^{\vee})\subset M_G$ is open and is the group of invertible elements, and $\pi_G^{-1}(S_G)\subset M_G$ is closed and is an absorbing subsemigroup.
\end{cor}

\begin{proof}
The only part left to check is that $\pi_G^{-1}(T^{\vee})$ consists of units. This follows from the fact that an endomorphism of the forgetful tensor functor 
$\Rep_{\mathbb{F}_p}(M_G) \rightarrow \Vect_{\bbF_p}$ is an automorphism if and only if it is an automorphism on simple objects.
\end{proof}

\appendix
\section{Cohomology with support in $T_{\nu}$} \label{appA}
Let $U^-$ be the unipotent radical of the opposite Borel $B^-$. For $\nu \in X_*(T)$, let $$T_\nu := (LU^- \cdot \nu(t))_{\text{red}} \subset \Gr_G$$ be the reduced ind-subscheme of the corresponding connected component of the \emph{repeller} (\cite{D13}) with respect to the $\mathbb{G}_m$-action on $\Gr_G$ from Subsection \ref{MVsec}.  For $\lambda\in X_*(T)^+$, we denote by $i_{T_{\nu, \lambda}} \colon T_\nu \cap \Gr_G^{\leq \lambda} \rightarrow \Gr_G^{\leq \lambda}$
% here we use the classical notation $\Gr_G^{\leq \lambda}$ instead of $\overline{\Gr_G^{\lambda}}$ since we do not need to emphasize on its properness nor embeddedness (contrary to what we deed in Sections 2 and 3)
the canonical immersion (where $T_\nu \cap \Gr_G^{\leq \lambda}$ is equipped with its reduced structure) and define 
$$
\forall i\in\bbZ,\quad H_{T_\nu}^i(\Gr_G, \IC_\lambda) : = R^i \Gamma (T_\nu \cap \Gr_G^{\leq \lambda},Ri_{T_{\nu, \lambda}}^! \IC_{\lambda}).
$$
\begin{prop*} \label{A1}
Let $\lambda\in X_*(T)^+$. If $\nu = w_0( \lambda)$ then 
$$
H^{2 \rho(\nu)}_{T_\nu}(\Gr_G, \IC_\lambda) = R^{-2 \rho(\lambda)}\Gamma(\IC_\lambda) \cong \mathbb{F}_p.
$$
\end{prop*}

\begin{proof}
By \cite[3.2]{GeometricSatake}, $T_\nu \cap \Gr_G^{\leq \lambda}$ is of pure dimension  $-\rho(\nu+w_0(\lambda))=2\rho(\lambda)=\dim \Gr_G^{\leq \lambda}$. Thus $T_\nu \cap \Gr_G^{\leq \lambda}$ is open in $\Gr_G^{\leq \lambda}$, so $Ri_{T_{\nu, \lambda}}^! = Ri_{T_{\nu, \lambda}}^*$ and the proposition follows. 
\end{proof}

\begin{prop*} \label{A2}
Let $\lambda\in X_*(T)^+$ be such that $\rho(\lambda) \neq 0$. If $\nu = \lambda$ then 
$$
H^{2 \rho(\nu)}_{T_\nu}(\Gr_G, \IC_\lambda) = 0.
$$ 
\end{prop*}

\begin{proof}
By \cite[3.2]{GeometricSatake}, $T_\nu \cap \Gr_G^{\leq \lambda}$ is a point.  Let $U := \Gr_G^{\leq \lambda}\setminus(T_\nu \cap \Gr_G^{\leq \lambda})$ and $j \colon U \rightarrow \Gr_{\leq \lambda}$ be the canonical open immersion. We claim that, as a complex of sheaves, $Rj_* \mathcal{O}_U$ is concentrated in degrees $\leq 2\rho(\lambda) - 1$. To prove the claim, note that we may replace $\Gr_G^{\leq \lambda}$ by the local ring $(A, \mathfrak{m})$ at $T_\nu \cap \Gr_G^{\leq \lambda}$ in $\Gr_G^{\leq \lambda}$. For $n \geq 1$ we have $R^n j_* \mathcal{O}_U = H^{n+1}_{\mathfrak{m}}(A)$. Since $H^{i}_{\mathfrak{m}}(A) = 0$ for $i> \dim A$ then $R^n j_* \mathcal{O}_U = 0$ unless $n \leq  2\rho(\lambda) - 1$.

 Now by the Artin-Schreier sequence $Rj_* (\mathbb{F}_p[2 \rho(\lambda)])$ is concentrated in degrees $\leq 0$. Hence by the exact triangle 
$$
\xymatrix{
Ri_{T_{\nu, \lambda}, *} Ri_{T_{\nu, \lambda}}^! (\mathbb{F}_p[2 \rho(\lambda)]) \ar[r] & \mathbb{F}_p[2 \rho(\lambda)] \ar[r] & Rj_* Rj^* (\mathbb{F}_p[2 \rho(\lambda)])  \ar[r]^<<<<<{+1} &.
}
$$ 
it follows that $Ri_{T_{\nu, \lambda}}^! (\mathbb{F}_p[2 \rho(\lambda)])$ is concentrated in degrees $\leq 1$. Now we are done because $2 \rho (\lambda) > 1$ and $T_\nu \cap \Gr_G^{\leq \lambda}$ is a point.
\end{proof}

\begin{prop*} \label{localizationfails}
Suppose $G = \SL_2$, and that $T$ and $B$ are the diagonal maximal torus and the upper triangular Borel subgroup; in particular $X_*(T)^+ \cong \mathbb{Z}_{\geq 0}$. If $\lambda = 1$ and $\nu = 0$, then $H_{T_\nu}^{2 \rho(\nu)}(\Gr_G, \IC_\lambda)$ is infinite-dimensional. 
\end{prop*}

\begin{proof}
The scheme $\Gr_G^{\leq \lambda}$ is stratified by $T_{-\lambda} \cap \Gr_G^{\leq \lambda}$, $T_\nu \cap \Gr_G^{\leq \lambda}$, and $T_{\lambda} \cap \Gr_G^{\leq \lambda}$. These strata have dimensions $2$, $1$, and $0$, respectively. Let $Z = \overline{T}_{\nu} \cap \Gr_G^{\leq \lambda} =  (T_\nu \cap \Gr_G^{\leq \lambda}) \cup (T_{\lambda} \cap \Gr_G^{\leq \lambda})$ and let $i \colon Z \rightarrow \Gr_G^{\leq \lambda}$ be the corresponding closed immersion. Let $j \colon T_{-\lambda} \cap \Gr_G^{\leq \lambda} \rightarrow \Gr_G^{\leq \lambda}$ be the complementary open immersion. Then there is an exact triangle 
$$
\xymatrix{
Ri_*Ri^!(\IC_\lambda) \ar[r] & \IC_\lambda \ar[r] & Rj_*Rj^*(\IC_\lambda)  \ar[r]^<<<<<{+1} &.
}
$$
By \cite[6.9]{modpGr}, $R\Gamma(\IC_\lambda) \cong \mathbb{F}_p[2]$, and the map $R^{-2}\Gamma(\IC_\lambda) \rightarrow R^{-2}\Gamma(Rj_*Rj^*(\IC_\lambda))$ is an isomorphism. By \cite[5.2]{ngopolo}, $T_{-\lambda} \cap \Gr_G^{\leq \lambda}$ is isomorphic to $\mathbb{A}^2$. Thus by a computation with the Artin-Schreier sequence we find that $R\Gamma(Rj_*Rj^*(\IC_\lambda))$ is concentrated in degrees $-2$ and $-1$, and it is infinite-dimensional in degree $-1$. Hence $R\Gamma(Ri_*Ri^!(\IC_\lambda))$ is concentrated in degree $0$, and $R^0\Gamma(Ri_*Ri^!(\IC_\lambda))$ is infinite-dimensional.

Now let $p \colon T_{\nu} \cap \Gr_{\leq \lambda} \rightarrow Z$ be the open immersion. There is an exact triangle
$$
\xymatrix{
R\Gamma(Ri_{T_{\lambda}}^! (\IC_{\lambda})) \ar[r] & R\Gamma(Ri^!(\IC_\lambda)) \ar[r] & R\Gamma(Rp^*(Ri^!(\IC_\lambda)))  \ar[r]^<<<<<{+1} &.
}
$$ 
By \cite[2.10]{modpGr}, $\IC_\lambda$ is the intermediate extension of its restriction to $\Gr_G^{\leq \lambda}\setminus(T_{\lambda} \cap \Gr_G^{\leq \lambda})$, so by \cite[2.7]{modpGr} $Ri_{T_{\lambda}}^! (\IC_{\lambda})$ is concentrated in degrees $\geq 1$. 
%(recall that here $T_{\lambda} \cap \Gr_G^{\leq \lambda}$ is a point ...!)
Thus the map $R^0\Gamma(Ri^!(\IC_\lambda)) \rightarrow R^0\Gamma(Rp^*(Ri^!(\IC_\lambda)))$ is injective. Now we are done because there is a natural isomorphism $R^0\Gamma(Rp^*(Ri^!(\IC_\lambda))) \cong H_{T_\nu}^{2 \rho(\nu)}(\Gr_G, \IC_\lambda)$.
\end{proof}

By comparing \ref{A1} and \ref{A2} with \ref{computetop}, we see that the groups
$$
H_c^{2 \rho(\nu)}(S_{\nu},\IC_{\lambda})\quad\textrm{and}\quad H_{T_{\nu}}^{2 \rho(\nu)}(\Gr_G,\IC_{\lambda})
$$ agree in some cases. However, by \ref{localizationfails}, these groups are not isomorphic in general. In other words, \emph{Braden's hyperbolic localization theorem fails for $\bbF_p$-coefficients}.

\end{document}